\documentclass[12pt]{amsart}
\usepackage{amsmath, amssymb}
\usepackage[usenames]{color}

\usepackage[francais,english]{babel}
\usepackage[pdftex]{graphicx,graphics}
\graphicspath{ {pictures/} }

\usepackage[utf8]{inputenc} 
\usepackage[colorlinks,linkcolor = blue]
{hyperref}

\newtheorem{Def}{Definition}[section]
\newtheorem{Thm}{Theorem}[section]
\newtheorem{Pro}{Proposition}[section]
\newtheorem{Lem}{Lemma}[section]

\theoremstyle{definition}
\newtheorem{Remark}{{\bf Remark}}[section]

\newcommand{\R}{\mathbb R}

\begin{document}

\title[ A \textit{priori} Lipschitz estimates for solutions of local and nonlocal Hamilton-Jacobi equations with Ornstein-Uhlenbeck operator]
{A \textit{priori} Lipschitz estimates for solutions of local and nonlocal Hamilton-Jacobi equations with Ornstein-Uhlenbeck operator}

\author{Emmanuel Chasseigne, Olivier Ley \& Thi Tuyen Nguyen}
\address{LMPT, Universit\'e Fran\c{c}ois-Rabelais, Tours, France} \email{Emmanuel.Chasseigne@lmpt.univ-tours.fr}
\address{IRMAR, INSA de Rennes \& Universit\'e de Rennes 1, Rennes, France} \email{olivier.ley@insa-rennes.fr,
thi-tuyen.nguyen@univ-rennes1.fr}

\begin{abstract}
  We establish \textit{a priori} Lipschitz estimates for unbounded solutions of second-order
  Hamilton-Jacobi equations in $\R^N$ in presence of an Ornstein-Uhlenbeck drift. We generalize the results obtained by Fujita, Ishii \& Loreti (2006) in several directions. The first one is to
consider more general operators.
We first replace the Laplacian by a general diffusion matrix and then consider
a nonlocal integro-differential operator of fractional Laplacian type. The 
second kind of extension is to deal with more general Hamiltonians which
are merely sublinear. These results are obtained for both degenerate and nondegenerate equations.
\end{abstract}


\date{\today}
\maketitle

\section{Introduction}\label{chapter-1}

We are concerned with \textit{a priori} Lipschitz estimates
for continuous unbounded viscosity solutions of the Hamilton-Jacobi equations
\begin{eqnarray}\label{aperg}
&& \lambda u^\lambda - \mathcal{F}(x,[u^\lambda])
+\langle b,Du^\lambda\rangle + H(x,Du^\lambda )=f(x),
\qquad x\in \R^N, \ \lambda >0,
\end{eqnarray}
and
\begin{eqnarray}\label{cauchy}
  && \begin{cases}
       \displaystyle
       \frac{\partial u}{\partial t} - \mathcal{F}(x,[u]) + \langle b(x), Du\rangle + H(x,Du) = f(x)
       \qquad (x,t)\in \R^N \times (0,\infty),\\
u(\cdot,0) = u_0(\cdot) ~~ \text{in } \R^N,
\end{cases}
\end{eqnarray}
where
$b$ is an {\em Ornstein-Uhlenbeck drift}, i.e., there exists $\alpha>0$ (the strength
of the Ornstein-Uhlenbeck term) such that
\begin{eqnarray}\label{dissipative}
\langle b(x) - b(y), x-y \rangle \geq \alpha |x-y|^2, \qquad x,y\in\R^N,
\end{eqnarray} 
the Hamiltonian $H$ is continuous and {\em sublinear}, i.e., there exists $C_H>0$ such that
\begin{eqnarray} \label{hypH}
|H(x,p)|\leq C_H(1+|p|), \qquad x,p\in\R^N,
\end{eqnarray}
and the operator $\mathcal{F}$ can be either \textit{local}
\begin{eqnarray} \label{F-diffusion}
\mathcal{F}(x,[u]) = {\rm tr}(A(x)D^2u) \qquad \text{(classical diffusion)}
\end{eqnarray}
where $A$ is a nonnegative symmetric matrix, or \textit{nonlocal}
\begin{eqnarray} \label{F-nonlocal}
&& \mathcal{F}(x,[u]) = \int_{\R^N} \{ u(x+z) - u(x) - \langle Du (x), z \rangle \mathbb{I}_{B}(z)\}\nu(dz) \qquad \text{(integro-differential)}.
\end{eqnarray}
\smallskip

More precise assumptions will be given below. In particular, the growth of the datas and the solutions
is a crucial point when considering such equations stated in the whole space $\R^N.$
It is why the expected Lipschitz bounds for the solutions of~\eqref{aperg},~\eqref{cauchy} are
\begin{eqnarray} \label{lip-bounds}
  |u^\lambda(x)-u^\lambda(y)|, \ |u(x,t)-u(y,t)| &\leq &
  C(\phi_\mu(x) + \phi_\mu(y))|x-y|, \qquad x,y\in\R^N,
\end{eqnarray}
providing that the datas $f$ and $u_0$ satisfy the same type of estimates
\begin{eqnarray}\label{fil-f-term}
|g(x) - g(y)| \leq C_g (\phi_\mu(x) + \phi_\mu(y))|x-y|, \qquad \text{$g=f$ or $g=u_0.$}
\end{eqnarray}
The continuous function $\phi_\mu$ takes into account the growth of the datas $f,u_0$ and the
solutions $u^\lambda, u(\cdot,t).$ Let us underline that we are looking for a constant $C$
which is {\em independent} of $\lambda, t>0,$ since our main motivation to establish such kind of bounds
is to apply them to solve some ergodic problems and to study the large time behavior of
the solutions of~\eqref{cauchy}. These issues will be discussed below.
\smallskip

In the particular case of the local equation with a pure Laplacian diffusion
and a Lipschitz continuous Hamiltonian $H$ independent of $x,$
\begin{eqnarray}\label{fil-cauchy}
  \frac{\partial u}{\partial t}      
- \Delta u +  \alpha\langle x, Du\rangle + H(Du) = f(x), 
\end{eqnarray}
Fujita, Ishii \& Loreti \cite{fil06} established the estimates~\eqref{lip-bounds}
for datas $f,u_0$ and solutions $u^\lambda, u(\cdot,t)$ belonging to the class
\begin{eqnarray}\label{fil-class-growth}
\mathcal{E}_\mu =\left\{ g: \R^N \to \R :
\mathop{\rm lim}_{|x|\to +\infty} \frac{g(x)}{\phi_\mu(x)} = 0\right\},
\end{eqnarray}
where
\begin{eqnarray}\label{fil-growth}
\phi_\mu(x) = e^{\mu|x|^2},
\end{eqnarray}
with $\mu <\alpha$, which seems to
be the optimal growth condition when thinking of the
classical heat equation.
\smallskip

The main result of this work is to
prove estimates~\eqref{lip-bounds} for solutions
of the general equations~\eqref{aperg},~\eqref{cauchy}
belonging to $\mathcal{E}_\mu$ for every $\mu >0,$
with
\begin{eqnarray}\label{exp-growth}
\phi_\mu(x) = e^{\mu\sqrt{1+|x|^2}}.
\end{eqnarray}
The Hamiltonian $H$ is continuous and merely sublinear (see~\eqref{hypH})
without further assumptions,
which allows to deal with general Hamiltonians of Bellman-Isacs-type
coming from optimal control and differential games. 
The datas $f,u_0$ satisfy~\eqref{fil-f-term}.
In the local case, the diffusion is anisotropic, the matrix $A$ can be written
$A=\sigma \sigma ^T$ where $\sigma\in W^{1,\infty}(\R^N; \mathcal{M}_N),$
i.e,
\begin{eqnarray}\label{hyp-sig}
  |\sigma(x)|\leq C_\sigma, \quad
  |\sigma(x)-\sigma(y)|\leq L_\sigma|x-y|
  \qquad x,y\in\R^N.
\end{eqnarray}
In the nonlocal case, $\mathcal{F}$ has the form~\eqref{F-nonlocal},
where $\nu$ is a L\'evy type measure, which is regular and nonnegative.
In order that~\eqref{F-nonlocal} is well-defined for our solutions
in $\mathcal{E}_\mu,$
\begin{eqnarray*}\label{non-oper}
  \mathcal{I}(x,\psi,D\psi)
  := \int_{\R^N} \{ \psi(x+z) - \psi(x) - \langle D\psi (x), z \rangle \mathbb{I}_{B}(z)\}\nu(dz)
\end{eqnarray*}
has to be well-defined for any continuous $\psi\in \mathcal{E}_\mu$ which is $C^2$ in
a neighborhood of $x,$ which leads to assume that
\begin{eqnarray}\label{M1}
  \left\{
  \begin{array}{l}
\text{There exists a constant $C^1_\nu > 0$ such that}\\
\displaystyle \int_{B}|z|^2\nu(dz) , \  \int_{B^c} \phi_\mu(z) \nu(dz) \ \leq C^1_\nu.
  \end{array}
  \right.
\end{eqnarray}
An important example of $\nu$ is the \textit{tempered $\beta$-stable law}
\begin{eqnarray}\label{nu}
\nu(dz)= \frac{e^{-\mu|z|}}{|z|^{N+\beta}}dz,
\end{eqnarray}
where $\beta \in (0,2)$ is the {\em order} of the integro-differential operator.
Notice that, in the bounded framework when $\mu$ can be taken equal to 0,
up to a normalizing constant, 
$-\mathcal{I} = (-\Delta)^{\beta/2}$ is the fractional Laplacian of order $\beta$, see \cite{dpv12}.
\smallskip

The restriction of the growth~\eqref{exp-growth} when comparing with~\eqref{fil-growth}
is due to ``bad'' first-order nonlinearities coming 
from the dependence of $H$ with respect to $x$ and the possible anistropy
of the higher-order operators.
These terms are delicate to treat in this unbounded setting. We do not know if the
growth~\eqref{exp-growth} is optimal. 
\smallskip

As far as Lipschitz regularity results are concerned, there is an extensive literature on
the subject. But most of them are local estimates or global estimates
but for bounded solutions in presence of a strongly coercive Hamiltonian.
In the case of a local diffusion,
local Lipschitz estimates for classical solutions are often obtained
via Bernstein's method, see Gilbarg-Trudinger~\cite{gt83} and Barles~\cite{barles91a} for a weak method in the
context of viscosity solutions. For strictly elliptic equations, Ishii-Lions~\cite{il90} developed
a powerful method we use in this work. See also
\cite{lions82, bs01, ls05, clp10, ln16, ln17} and the references therein. 
In the nonlocal setting, an important work is Barles et al.~\cite{bcci12} where Ishii-Lions'
method is extended for bounded solutions to strictly elliptic (in a suitable sense) integro-differential equations
in the whole space. See also~\cite{bcci14, bt16b, bklt15}, and~\cite{blt17}
for an extension of the Bernstein method in the nonlocal case with coercive Hamiltonian.
\smallskip

When working in the whole space with unbounded solutions, one need to recover
some compactness properties. It is the effect of the Ornstein-Uhlenbeck operator term.
In the case $-\Delta + \alpha \langle x, D\rangle,$
in terms of stochastic optimal control, the Laplacian is related to the usual diffusion which spreads
the trajectories while the term $\alpha\langle  x,D\rangle$ tends to confine the trajectories
near the origin allowing to recover some compactness.
From a PDE point of view, this property translates into a supersolution property
for the growth function $\phi_\mu,$ that is, there exists $C,K>0$ such that
\begin{eqnarray}\label{fil-crucial-tool}
  && \mathcal{L}[\phi_\mu](x):=  - \mathcal{F}(x,[\phi_\mu]) + \langle b(x), D\phi_\mu(x)\rangle
  - C|D\phi_\mu(x)| \geq \phi_\mu(x) - K,
\qquad x \in \R^N.
\end{eqnarray}
This property is the crucial tool used in~\cite{fil06} to prove~\eqref{lip-bounds}
and the existence and uniqueness of solutions for~\eqref{fil-cauchy}.
Let us also mention the works of
Bardi-Cesaroni-Ghilli~\cite{bcg15} and Ghilli~\cite{ghilli16} for local equations,
where~\eqref{lip-bounds}
are obtained for constant nondegenerate diffusions and bounded solutions but
for equations with possibly quadratic coercive Hamiltonians.
\smallskip

The totally degenerate case (i.e., without second order term in~\eqref{fil-cauchy})
is investigated in~\cite{fil06b, fl09}. This means  that Lipschitz regularity results
can be obtained without ellipticity in the equation and come directly from
the Ornstein-Uhlenbeck term.
Actually, one can already notice that, in~\cite{fil06}, the ellipticity of $-\Delta$
is used only for being able to work with classical solutions to~\eqref{fil-cauchy}
thanks to Schauder theory and to simplify the proofs.
In our work, contrary to~\cite{fil06, fil06b, fl09}  and due to the more general
equations~\eqref{aperg},~\eqref{cauchy}, non-degeneracy of the equation
is crucial to obtain our estimates~\eqref{lip-bounds}.
However, we present also some estimates for degenerate equations.
\smallskip

More precisely, we call the equations~\eqref{aperg},~\eqref{cauchy} nondegenerate
when
\begin{eqnarray}\label{ellipticite1}
A(x)\geq \rho Id, \quad \text{for some $\rho>0$},
\end{eqnarray}
holds in the local case, which is the usual ellipticity assumption. 
In the nonlocal case, there is no classical definition of ellipticity.
We have then to state one useful for our purpose.
We will work with L\'evy measures $\nu$ 
satisfying~\eqref{M1} and
\begin{eqnarray}\label{M2-assumpt}
  \left\{
\begin{array}{c}
\text{There exists $\beta \in (0,2)$ such that for every $a \in \R^N$ there exist } \\ 
\text{ $0 < \eta < 1$ and $C^2_\nu>0$ such that, for all $\gamma >0,$}\\
\displaystyle \int_{\mathcal{C}_{\eta,\gamma}(a)}|z|^2 \nu(dz) \geq C^2_\nu \eta^{\frac{N-1}{2}}\gamma^{2-\beta},
\end{array}
\right.
\end{eqnarray}
where $\mathcal{C}_{\eta,\gamma}(a) : = \{z \in B_\gamma: (1 - \eta)|z||a| \leq |\langle a,z \rangle| \}$
is the cone illustrated in Figure~\ref{dessin_cone}.
\begin{figure}[ht]
\begin{center}
\includegraphics[width=5cm]{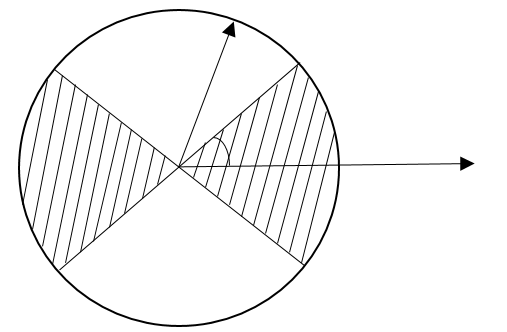}%
  \end{center}
\unitlength=1pt
\begin{picture}(0,0)%
\put(70,65){$\overrightarrow{a}$}
\put(0,112){$\gamma$}
\put(-23,50){$0$}
\put(0,68){$\theta$}
\end{picture}
\caption{\label{dessin_cone}The truncated cone $\mathcal{C}_{\eta,\gamma}(a)$} of axis $a$ and aperture $\theta$ with $\cos \theta = 1 - \eta(a).$
\end{figure}
This assumption, which holds true for the typical example~\eqref{nu}, is a kind of ellipticity condition
which was introduced in~\cite{bcci12} to adapt Ishii-Lions' method to nonlocal integro-differential
equations. In our unbounded case, this ellipticity property is not powerful enough
to control the first-order nonlinearities in the whole range
of order $\beta$ of the integro-differential operator but only
for $\beta\in (1,2)$ (as already noticed in~\cite{bcci12} for instance).
It is why,
 hereafter, to state in a convenient way our results, we will say 
that the nonlocal equation is {\em nondegenerate} when~\eqref{M2-assumpt} holds with $\beta\in\left(1,2\right)$.
\smallskip

Moreover, we also investigate~\eqref{lip-bounds} for solutions of degenerate equations, i.e.,
when the ellipticity condition~\eqref{ellipticite1} does not necessarily hold
and when $\beta\in (0,1].$ To obtain~\eqref{lip-bounds} in this framework,
we need to strengthen the hypotheses on $H$, that is
\begin{eqnarray}
\begin{cases}
\label{hypH2}
\text{ There exist } L_{1H},L_{2H} >0 \text{ such that for all } x,y,p,q\in\R^N \\
 |H(x,p)-H(y,p)|\leq L_{1H}|x-y|(1+|p|), \\
 |H(x,p)-H(x,q)| \leq L_{2H}|p-q|(1+|x|).
\end{cases}
 \end{eqnarray}
This is the classical assumption satisfied by a Hamiltonian coming from an optimal control problem.
\smallskip

Let us emphasize that the
Lipschitz estimates~\eqref{lip-bounds} are independent of $\lambda, t>0.$
The main application is the large time behavior
of the solutions of~\eqref{cauchy}, see~\cite{bs01, fil06, fil06b, fl09, bcci12, bcci14, blt17, ln16, ln17}
for instance. The fact that~\eqref{lip-bounds} is independent of $\lambda$ allows to send
$\lambda$ to 0 in~\eqref{aperg} and to solve the so-called ergodic problem associated
with~\eqref{cauchy}. Estimate~\eqref{lip-bounds} for~\eqref{cauchy} gives a compactness property
for the solution $u$ allowing to study the convergence of $u(x,t)$
as $t\to +\infty$ using the key property~\eqref{fil-crucial-tool} together with a strong maximum principle.
This program is carried out in~\cite{nguyen17}.
\smallskip

We end the introduction by giving a rough idea of the proof of~\eqref{lip-bounds}
in the stationary nondegenerate case.
The goal is to prove that
\begin{eqnarray}\label{max-intro}
&&  \mathop{\rm max}_{x,y\in\R^N} \{ u^\lambda(x)-u^\lambda(y) - \varphi(x,y)\},
  \qquad \text{with } \varphi(x,y)= \psi(|x-y|)(\phi_\mu(x)+\phi_\mu(y)),
\end{eqnarray}
is nonpositive for some concave function $\psi$ as in Figure~\ref{fr:dess_psi}, implying easily~\eqref{lip-bounds}.
\begin{figure}[ht]
\begin{center}
\includegraphics[width=6cm]{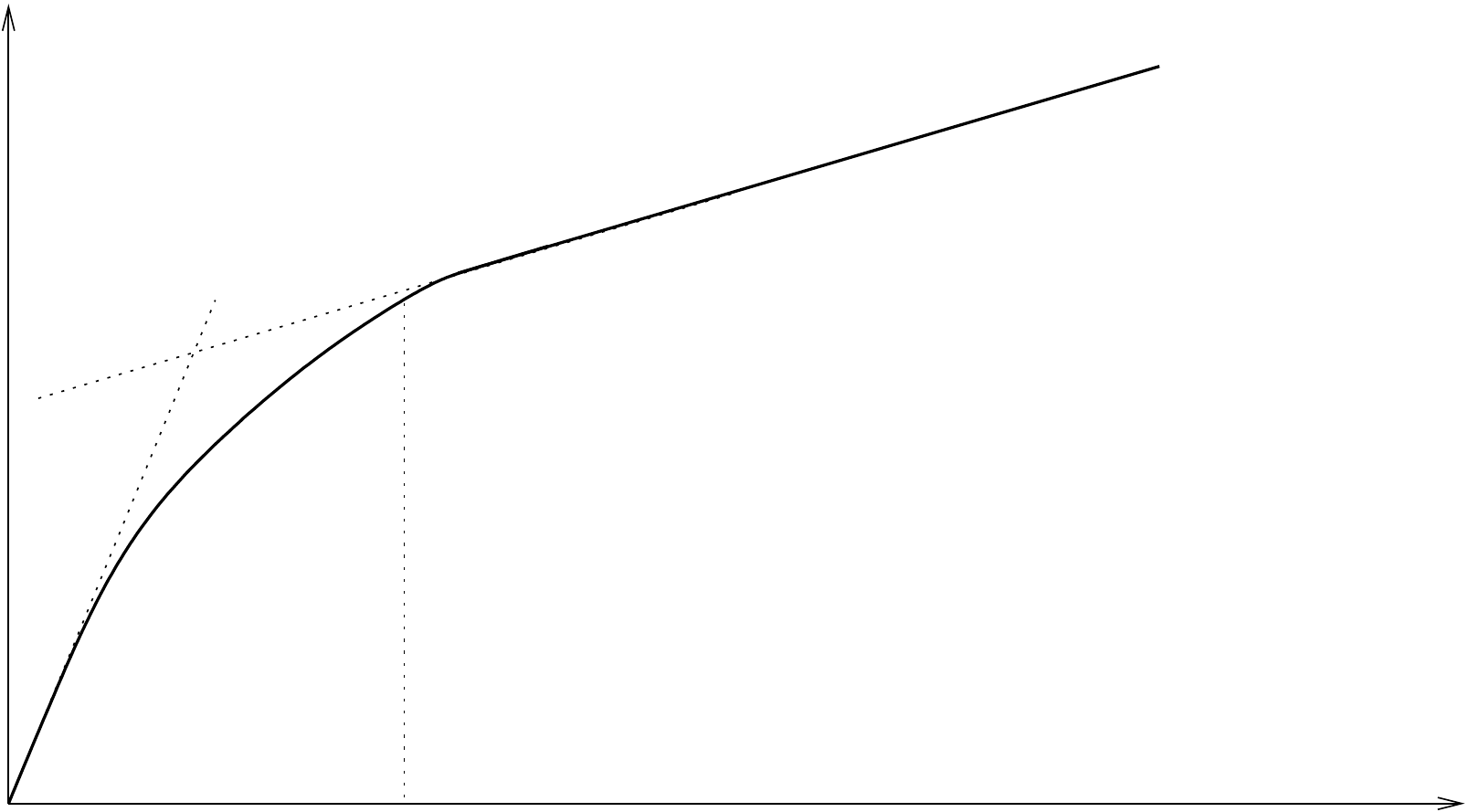}%
\end{center}
\unitlength=1pt
\begin{picture}(0,0)%
\put(-97,107){$\psi$}
\put(-40,10){$r_0$}
\put(88,10){$r$}
\end{picture}
\caption{\label{fr:dess_psi}The concave function $\psi$}
\end{figure}
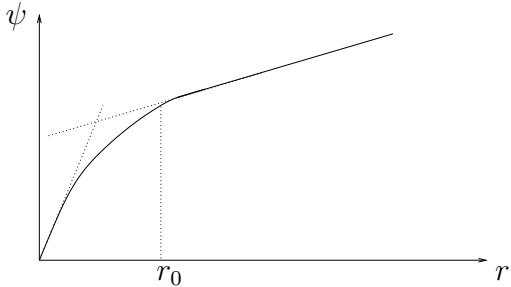

When writing down the sub/supersolution viscosity inequalities for $u^\lambda$
at $x$ and $y$, we are led to estimate several terms, some of which are ''good''
(the Ornstein-Ulhenbeck effect and the ellipticity of the equation)
while others are ``bad'' (first-order terms coming either from the
heterogeneity of the diffusion or the Hamiltonian). Notice that, when $H$ is Lipschitz
continuous and does not depend on $x$ as in~\cite{fil06}, the bad terms coming from $H$ are avoided.
\smallskip

Then, there are three ingredients that we use to derive estimate \eqref{max-intro}:
$(i)$ the first one consists in using the supersolution $\phi_\mu$ to control the
growth of different terms near infinity;
$(ii)$ for $|x-y|$ small we use the ellipticity of the diffusion
and we control the bad terms via Ishii-Lions' method (see \cite{il90, bcci12} and
Section~\ref{sec-ests}); $(iii)$ for $|x-y|$ big we control those terms by
the Ornstein-Uhlbenbeck drift~\eqref{dissipative}. 
\smallskip

In the degenerate case, we use the same approach, but with further assumptions 
on the datas and with conditions on the strength $\alpha$ of the
drift term to overcome the lack of ellipticity.
In the nonlocal case, although the main ideas are essentially the same, additional
tools and non-trivial adaptations are needed.
\smallskip


The paper is organized as follows. We establish the \textit{priori} Lipschitz estimates for the solutions
of~\eqref{aperg} both in the case of degenerate and nondegenerate equations in Section~\ref{sec-sta}.
The case of the parabolic equation~\eqref{cauchy} is investigated
in Section~\ref{sec-evo}. 
Finally Section~\ref{sec-ests} is devoted to some key estimates for the growth function
and the local and nonlocal operators. 
\smallskip

\noindent{\bf Notations and definitions.}
In the whole paper, $\mathcal{S}_N$ denotes the set of symmetric matrices of size $N$ equipped with
the norm $|A| = (\sum_{1\leq i,j\leq N} a_{ij}^2)^{1/2},$ $B(x,\delta)$ is the open ball of center $x$ and radius $\delta >0$
(written $B_\delta$ if $x=0$) and $B^c(x,\delta)=\R^N\setminus B(x,\delta)$.

Let $T \in (0,\infty)$, we write $Q = \R^N \times (0,\infty),$ $Q_T = \R^N \times[0,T)$
  and introduce the space $\mathcal{E}_\mu(\R^N)=\mathcal{E}_\mu$ (see~\ref{fil-class-growth}) and
\begin{eqnarray*}\label{emu}
\mathcal{E}_\mu(Q_T)&=&\left\{g: Q_T\to\R :
\mathop{\rm lim}_{|x|\to +\infty} \mathop{\rm sup}_{0\leq  t\leq T}\frac{g(x,t)}{\phi_\mu(x)} = 0\right\}.
\end{eqnarray*}
 Recall that $\phi_\mu$
is the growth function defined in~\eqref{exp-growth}. 
Throughout this paper we work with solutions which belong to these classes
and for simplicity, we will write indifferently
$\phi_\mu$ or $\phi$, $\mu>0$ being fixed. Notice that in the local case, we can
take $\mu>0$ arbitrary but in the nonlocal case, $\mu$ has to be chosen
so that $\phi_\mu$ is integrable with respect to the mesure $\nu$ outside some
ball, see \eqref{M1}.

In the whole article, we will deal with viscosity solutions of~\eqref{aperg},~\eqref{cauchy}.
Classical references in the local case are  \cite{cil92, koike04, fs93} and for
nonlocal integro-differential equations, we refer the reader to  \cite{bi08, at96}.
Since the definition is less usual in the nonlocal case, 
we recall it for~\eqref{aperg} (the same kind of definition holds for~\eqref{cauchy}
with easy adaptations). For $0 < \kappa \leq 1,$ we consider
\begin{eqnarray*}
&& \mathcal{I}[B_\kappa](x,u,p) = \int_{|z|\leq \kappa} [u(x+z) - u(x) - \langle p, z \rangle \mathbb{I}_B(z)]\nu(dz)\\
&&  \mathcal{I}[B^\kappa](x,u,p) = \int_{|z|> \kappa} [u(x+z) - u(x) - \langle p, z \rangle \mathbb{I}_B(z)]\nu(dz).
\end{eqnarray*}
\begin{Def}\label{def-vis-non}
  An upper semi-continuous (in short usc) function $u^\lambda \in \mathcal{E}_\mu(\R^N)$ is a subsolution
  of~\eqref{aperg} if for any $\psi \in C^2(\R^N)\cap \mathcal{E}_\mu(\R^N)$ such that $u^\lambda-\psi$ attains a maximum on $B(x,\kappa)$ at $x \in \R^N,$
\begin{eqnarray*}
\lambda u^\lambda (x) - \mathcal{I}[B_\kappa](x,\psi ,p) - \mathcal{I}[B^\kappa](x,u^\lambda ,p)
+\langle b(x),p\rangle + H(x,p)\leq f(x),
\end{eqnarray*}
where $p = D\psi(x)$, $ 0< \kappa \leq 1$.\\
An lower semi-continuous (in short lsc) function $u^\lambda \in \mathcal{E}_\mu(\R^N)$ is a supersolution of \eqref{aperg} if for any $\psi \in C^2(\R^N) \cap\mathcal{E}_\mu(\R^N)$ such that $u^\lambda-\psi$ attains a minimum on $B(x,\kappa)$ at $x \in \R^N$
\begin{eqnarray*}
\lambda u^\lambda (x) - \mathcal{I}[B_\kappa](x,\psi,p) - \mathcal{I}[B^\kappa](x,u^\lambda,p)
+\langle b(x),p\rangle + H(x,p)\geq f(x),
\end{eqnarray*}
where $p = D\psi(x)$, $0< \kappa \leq 1$.\\
Then, $u^\lambda$ is a viscosity solution of \eqref{aperg} if it is both a viscosity subsolution and a viscosity supersolution of \eqref{aperg}.
\end{Def}
%


%
\section{Regularity of solutions for stationary problem}\label{sec-sta}
\subsection{Regularity of solutions for uniformly elliptic equations.}
Recall that \eqref{aperg} is nondegenerate which means that the equation is strictly elliptic in the local case (i.e.,~\eqref{ellipticite1} holds) and \eqref{M2-assumpt} holds with $\beta \in (1,2)$ in the nonlocal one. In such a framework,
we can deal with merely sublinear Hamiltonians, i.e.,~\eqref{hypH} holds without further assumption. 

We state now the main result namely Lipschitz estimates which are uniform with respect to $\lambda >0,$ for the
solutions of \eqref{aperg}.

\begin{Thm}  \label{lip-er}
Let $u^{\lambda} \in C(\R^N)\cap \mathcal{E}_\mu(\R^N)$, $\mu >0,$ be a solution of \eqref{aperg}.  Assume~\eqref{dissipative},~\eqref{hypH} and \eqref{fil-f-term} for $f.$ Suppose in addition one of the following assumptions:
\\
\noindent{(i) $\mathcal{F}(x,[u^\lambda]) = {\rm tr}(A(x) D^2u^\lambda (x))$
and \eqref{hyp-sig}, \eqref{ellipticite1}  hold.}

\noindent{(ii)  $\mathcal{F}(x,[u^\lambda]) = \mathcal{I}(x,u^\lambda,Du^\lambda)$
and suppose that \eqref{M1} and \eqref{M2-assumpt} hold with $\beta \in (1,2)$.
}
\\
Then there exists a constant
$C$ independent of $\lambda$ such that
\begin{eqnarray}\label{estim-lip-u2}
&& |u^\lambda(x)-u^\lambda(y)|\leq C|x-y|(\phi_\mu(x)+\phi_\mu(y)), \quad x,y\in\R^N, \ \lambda \in (0,1).
\end{eqnarray}
\end{Thm}

\begin{proof}[{\bf Proof of Theorem \ref{lip-er}}]\text{ }

\smallskip

\noindent{\it 1. Test-function and maximum point.}
For simplicity, we skip the $\lambda$ superscript for $u^\lambda$ writing $u$ instead and we write $\phi$ for $\phi_\mu$.
Let $\delta, A,C_1 >0,$ $\psi:\R_+ \rightarrow \R_+$ with $\psi(0)=0$ be a $C^2$ concave and increasing function which will be defined later depending on the two different cases.

Consider
\begin{eqnarray}\label{def-sup}
&& M_{\delta,A,C_1}=\mathop{\rm sup}_{x,y\in\R^N}
\left\{
u(x)-u(y)-\sqrt{\delta}-C_1(\psi(|x-y|)+\delta)(\phi(x)+\phi(y)+A)
\right\}
\end{eqnarray}
and set
\begin{eqnarray}
&& \Phi(x,y)= C_1(\phi(x)+\phi(y)+A),~~ \varphi(x,y)=  \sqrt{\delta}+(\psi(|x-y|)+\delta)\Phi(x,y)\label{varphi}.
\end{eqnarray}
All the constants and functions will be chosen to be independent of $\lambda > 0$. 

We will prove that for any fixed $\lambda \in (0,1)$, there exists a $\delta_0(\lambda) >0$ such that for any $0< \delta < \delta_0(\lambda)$, $M_{\delta, A,C_1} \leq 0$.

Indeed, if $M_{\delta,A,C_1} \leq 0$ for some good choice of $A,C_1$,  $\psi$ independent
of $\delta >0,$ then we get \eqref{estim-lip-u2} by letting
$\delta\to 0$ . So we argue by contradiction,
assuming that for $\delta$ small enough $M_{\delta,A,C_1}>0.$
Since $u\in \mathcal{E}_\mu(\R^N)$ and $\delta > 0$, the supremum is achieved
at some point $(x,y)$ with $x \neq y$, thanks to $\delta >0$ and the continuity of $u$.
\smallskip

\noindent{\it 2. Viscosity inequalities.} 
We first compute derivatives of $\varphi$. For the sake of simplicity of notations, we omit $(x,y),$
we write $\psi$, $\Phi$ for $\psi(|x-y|),$ $\Phi(x,y)$ respectively.

Set
\begin{eqnarray} \label{p-C}
p=\frac{x-y}{|x-y|}, \qquad
\mathcal{C}= \frac{1}{|x-y|}(I-p\otimes p).
\end{eqnarray}
We have
\begin{eqnarray}
&& D_x\varphi = \psi' \Phi p + C_1(\psi+\delta)D\phi(x)\label{Dxphi}, \qquad D_y\varphi = - \psi' \Phi p + C_1(\psi+\delta)D\phi(y)\\
&& D^2_{xx}\varphi = \psi'' \Phi p\otimes p
     + \psi' \Phi \mathcal{C} +   C_1\psi' (p\otimes D\phi(x)+D\phi(x)\otimes p)  + C_1(\psi+\delta)D^2\phi(x)\nonumber\label{Dxxvarphi}\\
&& D^2_{yy}\varphi = \psi'' \Phi p\otimes p
     + \psi' \Phi \mathcal{C} -  C_1\psi' (p\otimes D\phi(y)+D\phi(y)\otimes p) + C_1(\psi+\delta)D^2\phi(y)\nonumber\label{Dyyvarphi}\\
&& D^2_{xy}\varphi = -\psi'' \Phi p\otimes p
     - \psi' \Phi \mathcal{C} +   C_1\psi' (D\phi(y)\otimes p - p \otimes D\phi(x))\nonumber\\
&& D^2_{yx}\varphi = -\psi''\Phi p\otimes p
     - \psi' \Phi \mathcal{C} +  C_1\psi' (p\otimes D\phi(y)-D\phi(x)\otimes p).
\nonumber
\end{eqnarray}

Then applying \cite[Theorem 3.2]{cil92} in the local case and \cite[Corollary 1]{bi08} in the nonlocal one we obtain, for any $\zeta >0,$
there exist $X ,Y\in \mathcal{S}^N$ such that 
$(D_x\varphi(x,y) ,X)\in \overline J^{2,+}u (x)$, 
$(-D_y\varphi(x,y), Y) \in \overline J^{2,-}u(y)$
and 
\begin{eqnarray}\label{ineg-matricielle}
\left(\begin{array}{cc}
X &O\\
O&-Y
\end{array}
\right)
\leq 
A + \zeta A^2, ~~ \text{ where}
\end{eqnarray}
\begin{eqnarray*}
A = D^2 \varphi(x,y)&=& 
\psi'' \Phi 
\left(\begin{array}{cc}
p\otimes p & -p\otimes p\\
-p\otimes p & p\otimes p
\end{array}
\right)
+ \psi' \Phi
\left(\begin{array}{cc}
\mathcal{C} & -\mathcal{C}\\
-\mathcal{C} & \mathcal{C}
\end{array}
\right)\\
&&+C_1 \psi' 
\left(\begin{array}{cc}
p\otimes D\phi(x)+D\phi(x)\otimes p & D\phi(y)\otimes p-p\otimes D\phi(x)\\
p\otimes D\phi(y)-D\phi(x)\otimes p   & -(p\otimes D\phi(y)+D\phi(y)\otimes p)
\end{array}
\right)\\
&&+ C_1(\psi+\delta)
\left(\begin{array}{cc}
D^2\phi(x) & 0\\
0 & D^2\phi(y)
\end{array}
\right)
\end{eqnarray*}
and $\zeta A^2 = O(\zeta)$ ($\zeta$ will be sent to $0$ first). 

Writing the viscosity inequality at $(x,y)$ in the local and nonlocal case we have
\begin{eqnarray}\label{ineq-visco}
&& \lambda (u(x)-u(y)) -\mathcal{F}(x,[u])+\mathcal{F}(y,[u]) \\ 
&& +\langle b(x),D_x\varphi\rangle - \langle b(y),-D_y\varphi\rangle + H(x,D_x\varphi) - H(y,-D_y\varphi)\nonumber\\
& \leq& f(x)-f(y),\nonumber
\end{eqnarray}
where $\mathcal{F}(x,[u]) = {\rm tr}(A(x)X)$ and $\mathcal{F}(y,[u]) = {\rm tr}(A(y)Y)$ in the local case and $\mathcal{F}(x,[u]) = \mathcal{I}(x,u,D_x \varphi)$ and $\mathcal{F}(y,[u]) = \mathcal{I}(y,u,-D_y \varphi)$ in the nonlocal one.

We estimate separately the different terms in order to reach a contradiction.
\smallskip

\noindent{\it 3. Monotonicity of the equation with respect to $u$.} 
Using that $M_{\delta, A, C_1}>0,$ we get
\begin{eqnarray} \label{term-en-lambda}
 \lambda (u(x)-u(y))> \lambda\sqrt{\delta}+\lambda (\psi+\delta)\Phi \geq \lambda\sqrt{\delta}.
\end{eqnarray}

\noindent{\it 4. Ornstein-Uhlenbeck-terms.}
From \eqref{dissipative} and \eqref{Dxphi} we have
\begin{eqnarray}\label{estim-b}
&&\langle b(x),D_x\varphi\rangle - \langle b(y),-D_y\varphi\rangle \\
&=& \psi'\Phi\langle b(x)-b(y),p\rangle
+ (\psi+\delta)(\langle b(x),D\phi(x)\rangle  
+\langle b(y), D\phi(y)\rangle)\nonumber\\
&\geq& 
\alpha \psi'\Phi |x-y|+ C_1(\psi+\delta)
(\langle b(x),D\phi(x)\rangle+\langle b(y),D\phi(y)\rangle).\nonumber
\end{eqnarray}

\noindent{\it 5. $H$-terms.} 
From \eqref{hypH} and \eqref{Dxphi}, we have
\begin{eqnarray}\label{estimH}
&& H(x,D_x\varphi) - H(y,-D_y\varphi)\geq   -C_H[2+ 2\psi' \Phi  + C_1(\psi+\delta)(|D\phi(x)| + |D\phi(y)|) ].
\end{eqnarray}

\noindent{\it 6. $f$-terms.} From~\eqref{fil-f-term}, we have
\begin{eqnarray}\label{estimf}
|f(x)-f(y)|\leq C_f(\phi(x)+\phi(y))|x-y|.
\end{eqnarray}
\noindent{\it 7. An estimate for the $\phi$-terms.}
%
To estimate the $\phi$-terms we use the following lemma the proof of which is postponed to Section \ref{sec-ests}.
\begin{Lem}\label{Lemphi}
Let $L_0 >0$, $L(x,y):= L_0(1 + |x| + |y|)$.
Define
\begin{eqnarray}\label{phi}
 \mathcal{L}_L[\phi](x,y) &:=& -\mathcal{F}(x,[\phi]) -\mathcal{F}(y,[\phi])+\langle b(x),D\phi(x)\rangle + \langle b(y),D\phi(y)\rangle\\
&&\qquad  - L(x,y)(|D\phi(x)| + |D\phi(y)|).\nonumber
\end{eqnarray}
There exists a constant $K=K(\alpha,L_0,\mathcal{F}) >0$
such that for any $\alpha > 2L_0,$
\begin{eqnarray}\label{etm-phi}
&& \mathcal{L}_L[\phi](x,y) \geq \phi(x) + \phi(y) - 2K. 
\end{eqnarray}
If $L(x,y) = L_0$, then \eqref{etm-phi} holds for any $\alpha > 0$ and there exists $R = R(\alpha,L_0,\mathcal{F})$ such that
\begin{eqnarray}\label{estim-phi2}
 -\mathcal{F}(x,[\phi])+\langle b(x),D\phi(x)\rangle -  L_0|D\phi(x)|\geq
\left\{
\begin{array}{ll}
- K & \text{for $|x|\leq R,$}\\
 K & \text{for $|x|\geq R.$}
\end{array}
\right. 
\end{eqnarray}
\end{Lem}
\noindent{\it 8. Global estimate from the viscosity inequality \eqref{ineq-visco}.} 
Plugging \eqref{term-en-lambda}, \eqref{estim-b}, \eqref{estimH} and \eqref{estimf} into \eqref{ineq-visco} we obtain
\begin{eqnarray}\label{global-new}
&&\lambda \sqrt{\delta} + \alpha \psi'(|x-y|)|x-y|\Phi(x,y) -\mathcal{F}(x,[u])+\mathcal{F}(y,[u])\\
&\leq& 2C_H \psi'(|x-y|)\Phi(x,y) + 2C_H + C_f(\phi(x) + \phi(y))|x-y|\nonumber\\
&& +  C_1(\psi(|x-y|)+\delta)\left(C_H(|D\phi(x)| + |D\phi(y)|) - \langle b(x),D\phi(x)\rangle -\langle b(y),D\phi(y) \rangle \right).\nonumber
\end{eqnarray}

The goal is now to reach a contradiction in \eqref{global-new}, first in the local case (whole step 9) and then in the nonlocal case (whole step 10).

\medskip
\noindent{\it 9. {\bf Local case:} Hypothesis $(i)$ holds, i.e., $\mathcal{F}(x,[u^\lambda]) = {\rm tr}(A(x) D^2u^\lambda (x))$ and \eqref{hyp-sig},~\eqref{ellipticite1} hold.}

\smallskip

\noindent{\it 9.1. Estimate for second order terms.} We use the following lemma the proof of which is given in Section \ref{sec-local}. 
\begin{Lem} \textbf{(Estimates on $\mathcal{F}$ in the local case).} \label{trace}\\
\noindent{(i) Degenerate case:
Under assumption \eqref{hyp-sig},}
\begin{eqnarray}\label{trace-deg}
&& {\rm - tr}(A(x)X - A(y)Y) \\
&\geq&  - \mathcal{C}_\sigma  |x-y|\psi'(|x-y|) \Phi(x,y) -C_1(\psi(|x-y|) + \delta) \{ {\rm tr}(A(x)D^2\phi(x))
\nonumber\\
&& +  {\rm tr}(A(y)D^2\phi(y))
 + \mathcal{C_\sigma}(|D\phi(x)|+|D\phi(y)|)\} +O(\zeta);\nonumber
\end{eqnarray}
\noindent{(ii) Elliptic case: In addition, if \eqref{ellipticite1} holds, we have}
\begin{eqnarray}\label{ineq-XY}
&&  {\rm - tr}(A(x)X - A(y)Y) \\
&\geq& 
- [4\rho \psi''(|x-y|) + \mathcal{C}_\sigma \psi'(|x-y|)]  \Phi(x,y)  - C_1\mathcal{C}_\sigma\psi'(|x-y|)(|D\phi(x)|+|D\phi(y)|)\nonumber \\
&& -C_1(\psi(|x-y|) + \delta)\{ {\rm tr}(A(x)D^2\phi(x)) +  {\rm tr}(A(y)D^2\phi(y))\nonumber\\
&& + \mathcal{C_\sigma}(|D\phi(x)|+|D\phi(y)|)\} +O(\zeta),\nonumber
\end{eqnarray}
where $ \mathcal{C}_\sigma = \mathcal{C}_\sigma(N, \rho)$ is given by \eqref{Csigma}.
\end{Lem}
%
This Lemma is a crucial tool giving the estimates for the second order terms. The first part is a basic application of Ishii's Lemma (see \cite{cil92}) in an unbounded context with the test function $\varphi$. The second part takes profit of the ellipticity of the equation and allows to apply Ishii-Lions' method (\cite{il90}).

\smallskip

\noindent{\it 9.2. Global estimate from \eqref{global-new}.}
Since \eqref{hyp-sig} and \eqref{ellipticite1} hold, using Lemma \ref{trace} \eqref{ineq-XY} to estimate for the difference of two local terms in \eqref{global-new}, letting $\zeta \to 0$ and using $|D\phi| \leq \mu \phi$  
we obtain
\begin{eqnarray}\label{ineg-contra124}
&& \lambda \sqrt{\delta}
+ C_1 (\phi(x)+\phi(y)+A)\left(-4\rho \psi''(|x-y|)  +\alpha\psi'(|x-y|) |x-y|\right)\\
\nonumber
&\leq& C_1 (\phi(x)+\phi(y)+A)(2C_H+(1+\mu)\mathcal{C}_\sigma)\psi'(|x-y|) + C_f(\phi(x)+\phi(y))|x-y|\\\nonumber
&& +2C_H  - C_1(\psi(|x-y|)+\delta)\mathcal{L}_L[\phi](x,y),\nonumber
\end{eqnarray}
where $\mathcal{L}_L$ is given by~\eqref{phi} with $\mathcal{F}$ introduced by \eqref{F-diffusion} and $L(x,y) = \mathcal{C}_{\sigma} + C_H$. 

The rest of the proof consists in reaching a contradiction in~\eqref{ineg-contra124}
by taking profit of the positive terms in the left hand-side of the inequality.
\smallskip

\noindent{\it 9.3. Construction of the concave test-function $\psi$.}
For $r_0, C_2>0$ to be fixed later, we define the $C^2$ concave
increasing function $\psi: [0,\infty)\to [0,\infty)$ as follows (see Figure \ref{fr:dess_psi})
\begin{eqnarray*}
\psi(r)= 1-e^{-C_2r} \quad \text{for $r\in [0,r_0],$}
\end{eqnarray*}
$\psi(r)$ is linear on $[r_0+1,+ \infty)$ with derivative
$\psi'(r)=C_2 e^{-C_2(r_0+1)},$ and $\psi$ is extended in a smooth
way on $[r_0,r_0+1]$ such that, for all $r\geq 0,$
\begin{eqnarray*}
\psi_{\rm min}':= C_2 e^{-C_2(r_0+1)}\leq \psi'(r)\leq \psi_{\rm max}':=\psi'(0)=C_2.
\end{eqnarray*}
Notice that $\psi$ is chosen such that
\begin{eqnarray}\label{edopsi}
\psi''(r)+C_2 \psi'(r)=0 \quad \text{for $r\in [0,r_0].$}
\end{eqnarray}

\noindent{\it 9.4. Choice of the parameters to reach a contradiction in~\eqref{ineg-contra124}.}
We now fix in a suitable way all the parameters to conclude that~\eqref{ineg-contra124}
cannot hold, which will end the proof. Before rigorous computations, let us explain roughly
the main ideas. We set $r:=|x-y|.$
The function $\psi$ above was chosen to be strictly concave for small $r\leq r_0.$
For such $r$ and for a suitable choice of $r_0$, we will take profit of the ellipticity of the equation, which appears
through the positive term $-4\nu\psi''(r)$ in~\eqref{ineg-contra124}, to control
all the others terms. Since we are in $\R^N$ and we cannot localize anything, $r$
may be large. In this case, the second derivative $\psi''(r)$
of the increasing concave function $\psi$ is small and the ellipticity is not
powerful enough to control the bad terms. Instead, we use the positive
term  $\alpha\psi'(r) r$ coming from the Ornstein-Uhlenbeck operator
to control everything for $r\geq r_0$.

At first, we set
\begin{eqnarray*}\label{defC1}
C_1= \frac{3C_f}{\alpha \psi_{\rm min}'} + 1=  \frac{3C_f e^{C_2 (r_0+1)}}{\alpha C_2} + 1,
\end{eqnarray*}
where $C_2$ and $r_0$ will be chosen later. This choice of $C_1$ is done in order
to get rid of the $f$-terms. Indeed, for every $r\in [0,+\infty),$
\begin{eqnarray}\label{debarasse-term-f}
&& C_1 (\phi(x)+\phi(y)+A)\frac{\alpha}{3}\psi'(r)r
\geq 
\frac{\alpha C_1 \psi_{\rm min}'}{3}(\phi(x)+\phi(y))r
\geq C_f(\phi(x)+\phi(y))r.
\end{eqnarray}

Secondly, we fix $r_0$ which separates the range of the ellipticity
action and the one of the Ornstein-Uhlenbeck term. We fix
\begin{eqnarray}\label{defr0}
r_0= \mathop{\rm max} \left\{\frac{3(2C_H+(1+\mu)\mathcal{C}_\sigma)}{\alpha}, 2R\right\},
\end{eqnarray}
where $R$ comes from \eqref{estim-phi2} with $L(x,y) = \mathcal{C}_\sigma + C_H$.
\smallskip

\noindent{\it 9.5. Contradiction in~\eqref{ineg-contra124} for $r\geq r_0$ thanks to the Ornstein-Uhlenbeck term.}
We assume that $r\geq r_0.$ With the choice of $r_0$ in \eqref{defr0} we have 
\begin{eqnarray*}
\frac{\alpha}{3}\psi'(r) r\geq \frac{\alpha}{3} \psi'(r) r_0\geq (2C_H+(1+\mu)\mathcal{C}_\sigma)\psi'(r).
\end{eqnarray*}

Moreover, $2R\leq r_0\leq r=|x-y|\leq |x|+|y|$ implies that either $|x|\geq R$ or 
$|y|\geq R,$ so by using \eqref{estim-phi2} we have $\mathcal{L}_L[\phi](x,y) \geq K-K\geq 0.$ Therefore, taking into account~\eqref{debarasse-term-f},
inequality~\eqref{ineg-contra124} reduces to
\begin{eqnarray*}
0< \lambda \sqrt{\delta}
+ C_1 (\phi(x)+\phi(y)+A)\frac{\alpha}{3}\psi'(r) r
\leq  2C_H.
\end{eqnarray*}
To obtain a contradiction, it is then sufficient to ensure
\begin{eqnarray}\label{choix1-A}
\frac{C_1A\alpha}{3} \psi_{\rm min}' r_0 \geq 2C_H \text{ which leads to } A\geq \frac{2C_H}{C_f r_0}
\end{eqnarray}
because of the choice of $C_1$ and the value of $\psi_{\rm min}'.$

Finally,~\eqref{ineg-contra124} can not hold for $r\geq r_0$
if $C_1, r_0, A$ are chosen as above. Notice that we did not impose
yet any condition on $C_2>0.$
\smallskip
\noindent{\it 9.6. Contradiction in~\eqref{ineg-contra124} for $r\leq r_0$ thanks to ellipticity.}
One of the main role of the ellipticity is to control the first term in the
right hand-side of~\eqref{ineg-contra124} for small $r.$ More precisely, by
setting
\begin{eqnarray}\label{choixC2-1}
C_2\geq \frac{2C_H+(1+\mu)\mathcal{C}_\sigma}{\rho },
\end{eqnarray}
and using~\eqref{edopsi}, we have
\begin{eqnarray*}
-\rho \psi''(r)\geq (2C_H+(1+\mu)\mathcal{C}_\sigma)\psi'(r).
\end{eqnarray*}

Since both $|x|$ and $|y|$ may be smaller than $R,$ we
cannot estimate $\mathcal{L}_L[\phi](x,y)$ from above in a better
way than $-2K.$ Taking into account~\eqref{debarasse-term-f},
inequality~\eqref{ineg-contra124} reduces to
\begin{eqnarray*}
0< \lambda \sqrt{\delta}
+ C_1 (\phi(x)+\phi(y)+A)(-3 \rho ) \psi''(r)
\leq  C_1(\psi(r)+\delta)2K +2C_H.
\end{eqnarray*}

Using that $\psi(r)\leq 1,$ and $C_1 \geq 1$,
we then increase $A$ from~\eqref{choix1-A} in order that
\begin{eqnarray*}
-3\rho A \psi''(r)\geq 2(K + C_H) \geq 2K \psi(r) + \frac{2C_H}{C_1},
\end{eqnarray*}
which leads to the choice
\begin{eqnarray*}\label{choix2-A}
A\geq \mathop{\rm max} \left\{\frac{2C_H}{C_f r_0} \, , \,  
\frac{2(K+C_H)e^{C_2r_0}}{3\rho C_2^2} \right\}.
\end{eqnarray*}

Notice that, by the choice of $r_0$ in~\eqref{defr0} 
and $C_2$ in~\eqref{choixC2-1},
$A$ and $C_1$ depend only on the datas $\sigma, \rho, b, H,f, \mu$
of the equation.

Finally, inequality~\eqref{ineg-contra124} becomes
\begin{eqnarray*}\label{choicedelta}
0< \lambda \sqrt{\delta}
\leq  2 C_1 K \delta,
\end{eqnarray*}
which is absurd for $\delta$ small enough.
It ensures the claim of Step 9.6.
\smallskip

\noindent{\it 9.7. Conclusion.}
We have proved that $M_{\delta,A,C_1} \leq 0$ in \eqref{def-sup}, for $\delta$ small enough.
It follows that, for every $x,y\in\R^N,$
\begin{eqnarray*}
|u(x)-u(y)|\leq \sqrt{\delta}+(\psi(|x-y|)+\delta)C_1(\phi(x)+\phi(y)+A).
\end{eqnarray*}
Since $A,\psi$ do not depend on $\delta,$ we can let $\delta\to 0$.
Using the fact that $\psi$ is a concave increasing function, we have 
$\psi(r) \leq \psi'(0)r = C_2r$, which leads to
\begin{eqnarray*}
|u(x)-u(y)|\leq C_1C_2 |x-y|(\phi(x)+\phi(y)+A)
= C_1C_2(A+1)|x-y|(\phi(x)+\phi(y)).
\end{eqnarray*}
This ends the proof of the local case $(i)$.

\medskip
\noindent{\it 10. {\bf Nonlocal case:} Hypothesis $(ii)$ holds, i.e.,  $\mathcal{F}(x,[u^\lambda]) = \mathcal{I}(x,u^\lambda, Du^\lambda)$.} In this case, it seems difficult to obtain directly the Lipschitz regularity as in the local case. Instead, we first establish $\tau$-H\"older continuity \textit{for all} $\tau \in (0,1)$ and then improve that H\"older regularity to Lipschitz regularity.
\smallskip

\noindent{\it 10.1. Nonlocal estimates for concave test functions.}
The following proposition will be crucial at several places in the proof
and is an adaptation of~\cite[Prop. 8]{bcci12}.
%
\begin{Pro}\label{nlocalgen} \textbf{(Concave estimates - general nonlocal operators).} Suppose that \eqref{M1} holds.  Let~$u \in C(\R^N)\cap \mathcal{E}_\mu(\R^N)$, we consider
$
\Psi(x,y) = 
u(x)- u(y)- \varphi(x,y), 
$
with $\varphi$ is defined in \eqref{varphi} using a concave function $\psi$. Assume the maximum of $\Psi$ is positive and reached at~$(x,y)$, with~$x \neq y$.
Let $a = x - y$ and choose any $a_0 > 0.$

\noindent{(i) \textbf{(Rough estimate for big $|a|$).}} For all $|a| \geq a_0$,
\begin{eqnarray}\label{degnon}
&& \mathcal{I}(x,u,D_x\varphi) - \mathcal{I}(y,u,-D_y\varphi)
\leq C_1(\psi(|a|) + \delta)\left(\mathcal{I}(x,\phi,D\phi) + \mathcal{I}(y,\phi,D\phi) \right).
\end{eqnarray}
\noindent{(ii) \textbf{(More precise estimate for small $|a|$).}}  For all $|a|\leq a_0$,
\begin{eqnarray}\label{ellnon}
  && \mathcal{I}(x,u,D_x\varphi) - \mathcal{I}(y,u,-D_y\varphi) \\
  \nonumber
  &\leq& C_1(\psi(|a|) + \delta)\left(\mathcal{I}(x,\phi,D\phi) + \mathcal{I}(y,\phi,D\phi) \right)
  + \mu C^1_\nu  \psi'(|a|)\Phi(x,y)\\
\nonumber
&&+\frac{1}{2}\Phi(x,y)  \int_{\mathcal{C}_{\eta, \gamma}(a)} \sup_{|s| \leq 1} \left\{
(1- \tilde{\eta}^2)\frac{\psi'(|a+sz|)}{|a+sz|} + \tilde{\eta}^2 \psi''(|a+sz|)
\right\}|z|^2\nu(dz),
\end{eqnarray}
where $\mathcal{C}_{\eta,\gamma}(a) = \{ z \in B_\gamma: (1- \eta)|z||a| \leq |\langle a,z \rangle|\}$ and $\gamma = \gamma_0|a|,$  $\tilde{\eta} = \frac{1 - \eta - \gamma_0}{1+\gamma_0}>0$ with $\gamma_0 \in (0,1),$ $\eta \in (0,1)$ small enough.
\end{Pro}
The interesting part in the estimate is the negative term $\psi''$.
It is integrated over $\mathcal{C}_{\eta,\gamma}(a),$ which gives a negative term
thanks to  Assumption~\eqref{M2-assumpt}. The magnitude of this ``good'' negative
term depends on the measure $\nu$ and the choice of the concave function $\psi.$
The proof of this Proposition will be given in Section \ref{sec-nonlocal}.
\smallskip

\noindent{\it 10.2. {Establishing $\tau$-H\"older continuity for all $\tau \in (0,1)$.}}
\begin{Pro} \label{holder-lemma} {\bf {(H\"older estimates).}} Suppose that \eqref{M1} and \eqref{M2-assumpt} hold with $\beta \in (1,2)$, consider \eqref{aperg} with $\mathcal{F}(x,[u^\lambda]) = \mathcal{I}(x,u^\lambda,Du^\lambda)$. Let~$\mu > 0$ and  $u^{\lambda} \in C(\R^N)\cap \mathcal{E}_\mu(\R^N)$ be a solution of~\eqref{aperg}.  Assume that~\eqref{dissipative},~\eqref{hypH} and~\eqref{fil-f-term} hold. Then for all $0< \tau <1$, there exists a constant
$C = C_\tau >0$ independent of $\lambda$ such that
\begin{eqnarray}\label{holder-estim-lip-u2}
&& |u^\lambda(x)-u^\lambda(y)|\leq C|x-y|^\tau(\phi(x)+\phi(y)), \quad x,y\in\R^N, \ \lambda \in (0,1).
\end{eqnarray}
\end{Pro}
This result is interesting by itself and can give H\"older regularity in a more slightly general framework than the one needed to establish Lipschitz regularity. It requires to build a specific concave test-function. The proof of this result is postponed after the one of the Theorem.

Notice that $\beta \in (1,2)$ here, see Remark \ref{rm-holder} for more explanation. We now resume the proof from the end of Step 8.
 \smallskip

\noindent{\it 10.3. Construction of the concave test-function $\psi$ to improve H\"older regularity to Lipschitz regularity.}
Let $\theta \in (0, \frac{\beta -1}{N+2-\beta})$, $\varrho = \varrho(\theta)$ be a constant such that $\varrho > \frac{2^{1-\theta}}{\theta}$
and $r_0 = r_0(\theta)>0$ such that
\begin{eqnarray}\label{r-small}
r_0 \leq  \frac{1}{2}
\left(\frac{1}{2\varrho(1+\theta)}\right)^{1/\theta}.
\end{eqnarray}
We then define a $C^2$ concave increasing function $\psi : [0,+ \infty) \rightarrow [0,\infty)$ such that
\begin{eqnarray}\label{psinon}
\psi(r)= r - \varrho r^{1+\theta}  \quad \text{for $r\in [0,r_0],$} 
\end{eqnarray}
$\psi(r)$ is linear on $[2r_0,+ \infty)$ with derivative
$\psi'(r)= 1- \varrho(1+\theta)(2r_0)^\theta,$ and $\psi$ is extended in a smooth
way on $[r_0,2r_0]$ such that, for all $r\geq 0,$
\begin{eqnarray}\label{psi'-Lip}
&& \frac{1}{2} \leq \psi_{\rm min}':= 1- \varrho(1+\theta)(2 r_0)^\theta \leq \psi'(r) \leq 1.
\end{eqnarray}
%

%

We continue the proof. We consider two cases. For $|x-y| = r \geq r_0$, we use the $\tau$-H\"older continuity of $u$ to get a contradiction directly without using the equation. For $|x-y| = r \leq r_0$ with $r_0$ small and fixed in \eqref{r-small}, we use the benefit of the ellipticity which comes from the nonlocal operator combined with the Ornstein-Uhlenbeck term to get a contradiction.
\smallskip

\noindent{\it 10.4. Reaching a contradiction for big $r = |x-y| \geq r_0$.}
Recalli that for all $\delta >0$ small we have
$$u(x) - u(y) > (\psi(r) + \delta)\Phi(x,y).$$
By the concavity of $\psi$ and the $\tau-$H\"older continuity of $u$, we have
\begin{eqnarray*}
&& r\psi'(r)\Phi(x,y) \leq \psi(r)\Phi(x,y) < u(x) - u(y) \leq C_\tau r^\tau(\phi(x) + \phi(y)), ~~ \tau \in (0,1).
\end{eqnarray*}
Since $\Phi(x,y) = C_1(\phi(x)+ \phi(y) + A),$ dividing the above inequalities by $\phi(x) + \phi(y) + A$, we get
$$C_1 \psi'(r) < C_\tau r^{\tau - 1} \leq C_\tau r_0^{\tau -1}, ~~ \tau \in (0,1), ~ r \geq r_0.$$
Moreover, from \eqref{psi'-Lip} we know that $\psi'(r) \geq \psi'_{\min} \geq \frac{1}{
2}$. Thus we only need to fix 
\begin{eqnarray}\label{choix-C1-new}
C_1 \geq 2C_\tau r_0^{\tau -1}.
\end{eqnarray}
Then we get a contradiction for all $r \geq r_0$.
\smallskip

\noindent{\it 10.5. Reaching a contradiction for small $r = |x-y| \leq r_0$.} 
Using the concave test function defined in \eqref{psinon} to Proposition \ref{nlocalgen} we get the following estimate for the difference of two nonlocal terms in \eqref{global-new}.

\begin{Lem}\label{Liplem}
Under the assumptions of Proposition \ref{nlocalgen}, let $\psi$ be defined by~\eqref{psinon}, $a_0 = r_0$ and $r= |x-y|  \leq r_0$. There exist $C(\nu), C^1_\nu > 0$ such that for $\Lambda = \Lambda(\nu) = C(\nu)[\varrho\theta2^{\theta-1}-1]>0$ we have
\begin{eqnarray}\label{ellnon-Lip}
&& \mathcal{I}(x,u,D_x\varphi) - \mathcal{I}(y,u,-D_y\varphi)\\
&\leq& C_1(\psi(r) + \delta)\left(\mathcal{I}(x,\phi,D\phi) + \mathcal{I}(y,\phi,D\phi) \right) - \left(\Lambda r^{-\tilde{\theta}} - \mu C^1_\nu \psi'(r) \right)\Phi(x,y),\nonumber 
\end{eqnarray}
where $\tilde{\theta} = \beta - 1 - \theta(N+2-\beta) > 0$.
\end{Lem}
The proof of this Lemma is given in Section \ref{sec-nonlocal}. 
\smallskip

Using Lemma \ref{Liplem} \eqref{ellnon-Lip}
into \eqref{global-new}, introducing $\mathcal{L}_L$ given by~\eqref{phi} with $L(x,y) = C_H$, $\mathcal{F}$ introduced by \eqref{F-nonlocal} and applying \eqref{estim-phi2},
we obtain
\begin{eqnarray}\label{ineg-contra1}
&& \lambda\sqrt{\delta}
+C_1(\phi(x) + \phi(y) + A)
(\Lambda r^{- \tilde{\theta}} + \alpha\psi'(r)r  )\\ 
&\leq&  C_1(\phi(x) + \phi(y) + A)\bar{C}\psi'(r)   + C_1(\psi(r) + \delta)2K + 2C_H + C_f(\phi(x)+\phi(y))r,\nonumber
\end{eqnarray}
where $\bar{C} = 2C_H  +\mu C^1_\nu,$ $\Lambda = C(\nu)(\varrho\theta 2^{\theta -1} - 1) > 0, ~~ \tilde{\theta} = \beta - 1 - \theta(N + 2 - \beta) > 0.$

We first increase $C_1$ in \eqref{choix-C1-new} as
\begin{eqnarray*}
C_1 = \max \left\{ \frac{C_f}{\alpha \psi'_{\min}} + 1,\  2C_\tau r_0^{\tau -1} \right\}
\end{eqnarray*}
 in order
to get rid of the $f$-terms. Indeed, for every $r\in [0,+\infty),$
\begin{eqnarray}\label{debarasse-term-f.2}
&& C_1 (\phi(x)+\phi(y)+A)\alpha\psi'(r)r
\geq
\alpha C_1 \psi_{\rm min}'(\phi(x)+\phi(y))r
\geq C_f(\phi(x)+\phi(y))r.
\end{eqnarray}
Then taking into account \eqref{debarasse-term-f.2} we get
\begin{eqnarray*}
\lambda \sqrt{\delta} + C_1(\phi(x) + \phi(y) + A) \Lambda r^{- \tilde{\theta}}\leq C_1(\phi(x) + \phi(y) + A)\bar{C}\psi'(r)
  + C_1(\psi(r) + \delta)2K + 2C_H.
\end{eqnarray*}
Moreover, since $\psi(r) \leq r \leq r_0$ and $C_1 \geq 1$, then we fix $A > 0$ in order that
\begin{eqnarray*}
\frac{1}{2}A \Lambda r^{- \tilde{\theta}}\geq \frac{1}{2}A \Lambda r_0^{-\tilde{\theta}} \geq 2(Kr_0 + C_H) \geq 2K\psi(r) + \frac{2C_H}{C_1},
\end{eqnarray*}
which leads to the choice
\begin{eqnarray*}
A \geq  \frac{4(Kr_0 + C_H)r_0^{\tilde{\theta}}}{\Lambda} .
\end{eqnarray*}
Then, inequality \eqref{ineg-contra1} becomes
\begin{eqnarray*}
\lambda \sqrt{\delta} + C_1(\phi(x) + \phi(y) + A)\frac{1}{2} \Lambda r^{- \tilde{\theta}} 
\leq C_1(\phi(x) + \phi(y) + A)\bar{C}\psi'(r)
  + 2KC_1 \delta.
\end{eqnarray*}

Now we fix $r_0$ satisfying \eqref{r-small} such that
$$ r_0 = \min \left\{ \left(\frac{\Lambda}{2\bar{C}}\right)^{1/\tilde{\theta}},\  \frac{1}{2} \left(\frac{1}{2\varrho(1+\theta)}\right)^{1/\theta}\right\}.$$
Then with the definition of $r_0$ as above, we get rid of $\bar{C}\psi'(r)$ term.

Finally, inequality \eqref{ineg-contra1} becomes
$$0 < \lambda \sqrt{\delta} \leq 2C_1 K \delta,$$
which is a contradiction for $\delta$ small enough. 
\smallskip

\noindent{\it 10.6. Conclusion:} We have proved that $M_{\delta,A,C_1} \leq 0$, for $\delta$ small enough.  For every $x,y \in \R^N$, 
$$|u(x) - u(y)| \leq \sqrt{\delta} + (\psi(|x-y|) + \delta)C_1 (\phi(x) + \phi(y) + A).$$
Since $A$ and $\psi$ do not depend on $\delta$, letting $\delta \to 0$ and using the fact that $\psi(r) \leq r$, we get
$$|u(x) - u(y)| \leq C_1|x-y|(\phi(x) + \phi(y) + A) \leq C_1(A+1)|x-y|(\phi(x) + \phi(y)).$$
Since $\phi \geq 1$, \eqref{estim-lip-u2} holds with $C = C_1(A+1).$ This concludes Step 10 and the proof of the Theorem in the nonlocal case $(ii)$.
\end{proof}

\begin{Remark}\label{rm-holder}
 For $\beta \in (0,1]$, the ellipticity  combined with the Ornstein-Uhlenbeck operator seems not powerful enough to control bad terms of the equation as in the local case. Therefore, in this case, we should need an additional condition on the strength $\alpha$ of the Ornstein-Uhlenbeck operator, see Theorem~\ref{lip-er-deg}.
\end{Remark}

\begin{proof}[\bf {Proof of Proposition \ref{holder-lemma} (Establishing $\tau$-H\"older continuity for all $\tau \in (0,1)$).}]
The beginning of the pr\-o\-of follows the lines of Step~1 to Step~8 in the p\-r\-o\-of of Theorem~\ref{lip-er}. We only need to construct a suitable concave increasing function (different from the one of Step~9.3) to get a contradiction in \eqref{global-new}.
\smallskip

\noindent{\it 1. Construction of the concave test-function $\psi$.} For $r_0, C_2>0$ to be fixed later, let $\tau \in (0,1)$ and define the $C^2$ concave
increasing function $\psi: [0,\infty)\to [0,\infty)$ as follows:
\begin{eqnarray}\label{holder-psi}
\psi(r)= 1-e^{-C_2r^\tau} \quad \text{for $r\in [0,r_0],$}
\end{eqnarray}
$\psi(r)$ is linear on $[r_0+1,+ \infty)$ with derivative
$\psi'(r)=C_2\tau (r_0+1)^{\tau-1} e^{-C_2(r_0+1)^{\tau}},$ and $\psi$ is extended in a smooth
way on $[r_0,r_0+1]$ such that, for all $r\geq 0,$
\begin{eqnarray*}\label{psi'}
\psi'(r) \geq \psi_{\rm min}':=C_2\tau (r_0+1)^{\tau-1} e^{-C_2(r_0+1)^{\tau}}.
\end{eqnarray*}

\noindent{\it 2. Global estimate to get a contradiction in \eqref{global-new}.} As in Step 9 of the Theorem, to estimate and reach a contradiction for \eqref{global-new} we need to separate the proof into two cases. For $r=|x-y|$ small we use the ellipticity coming from the nonlocal operator to control the other terms and for $r=|x-y|$ big enough, we can take the benefit of the Ornstein-Uhlenbeck term to control everything.
\smallskip

\noindent{\it 3. Reaching a contradiction in \eqref{global-new} when $|x-y|=r \geq r_0$ for a suitable choice of $r_0$.}
We first take $a_0 = r_0$ in Proposition \ref{nlocalgen} and use \eqref{degnon} to estimate the difference of nonlocal terms in \eqref{global-new}. Then \eqref{global-new} now becomes
\begin{eqnarray*}
&&\lambda\sqrt{\delta}
+\Phi(x,y)
\{\alpha\psi'(r)r - 2C_H\psi' (r) \}- 2{C_H}
+C_1(\psi(r)+\delta)\mathcal{L}_{ L}[\phi](x,y)\\
&\leq&
C_f(\phi(x)+\phi(y))r,\nonumber
\end{eqnarray*}
where $\mathcal{L}_L$ is the operator introduced by \eqref{phi} with  $L(x,y) = C_H,$ $\mathcal{F}$ is the nonlocal operator defined by \eqref{F-nonlocal}.

 We fix all constants in the same way that we did in Step 9 of the Theorem \ref{lip-er}. More precisely, we fix
\begin{eqnarray}
&& r_0 = \max \left\{ \frac{3(2C_H + \mu \hat{C}(\nu))}{\alpha} ; 2R\right\}, \qquad
C_1 = \frac{3C_f}{\alpha \psi'_{\min}} + 1, \qquad A \geq \frac{2C_H}{C_f r_0},\label{holder-choix1-A}
\end{eqnarray} 
where $R$ is a constant coming from \eqref{estim-phi2}. We use these choices of constants and the same arguments as those of Step 9.5  we get that \eqref{global-new} leads to a contradiction.
\smallskip

\noindent{\it 4. Reaching a contradiction in \eqref{global-new} when $|x-y|=r \leq r_0.$} 
In this case, we use the construction of $\psi$ in \eqref{holder-psi} applying to Proposition \ref{nlocalgen} \eqref{ellnon} to estimate the difference of the two nonlocal terms in \eqref{global-new}. This estimate is presented by following Lemma the proof of which is given in Section \ref{sec-nonlocal}.
\begin{Lem}\label{holder-lem}
Under the assumptions of Proposition \ref{nlocalgen}, let $\psi$ be a concave function defined by~\eqref{holder-psi}. Then for $0 < |x-y| = r \leq r_0$, there are constants $C^1_\nu, C(\nu,\tau)>0$ such that
\begin{eqnarray}\label{nonholder}
&&\mathcal{I}(x,u,D_x\varphi) - \mathcal{I}(y,u,-D_y\varphi)\\
&\leq& C_1(\psi(r) + \delta)\left(\mathcal{I}(x,\phi,D\phi) + \mathcal{I}(y,\phi,D\phi) \right)\nonumber\\
&& + \Phi(x,y)\left( \mu C^1_\nu - C(\nu,\tau)(1+C_2r^\tau)r^{1-\beta} \right)\psi'(r).\nonumber 
\end{eqnarray}
\end{Lem}
Applying Lemma \ref{holder-lem} \eqref{nonholder}
into \eqref{global-new}, introducing $\mathcal{L}_L$ given by~\eqref{phi} with $L(x,y) = C_H$ and $\mathcal{F}$ defined by \eqref{F-nonlocal},
we obtain
\begin{eqnarray}\label{holder-ineg-contra}
&& \lambda\sqrt{\delta}
+C_1(\phi(x) + \phi(y) + A)
\lbrace C(\nu,\tau)\psi'(r)(1+C_2 r^\tau)r^{1-\beta} + \alpha\psi'(r)r  \rbrace
\\
&\leq&  C_1(\phi(x) + \phi(y) + A)C(H,\nu,\mu)\psi'(r) - C_1(\psi(r) + \delta)\mathcal{L}_L[\phi](x,y) \nonumber\\
&&    + 2C_H + C_f(\phi(x)+\phi(y))r,\nonumber
\end{eqnarray}
where $C(H,\nu,\mu) = 2C_H + \mu C^1_\nu$. Since both $|x|$ and $|y|$ may be smaller than $R,$ we
cannot estimate $\mathcal{L}_L[\phi](x,y)$ from above in a better
way than $-2K.$ Taking into account~\eqref{debarasse-term-f},
inequality~\eqref{holder-ineg-contra} reduces to
\begin{eqnarray*}
&& \lambda\sqrt{\delta}
+C_1(\phi(x) + \phi(y) + A)
 C(\nu,\tau)\psi'(r)(1+C_2 r^\tau)r^{1-\beta}  
\nonumber\\
&\leq&  C_1(\phi(x) + \phi(y) + A)C(H,\nu,\mu)\psi'(r)  + C_1(\psi(r) + \delta)2K + 2C_H.
\end{eqnarray*}
Since $\tau < 1 < \beta$ then $\tau - \beta < 0$ and $r \leq r_0$, using that $\psi(r)\leq 1$ and $C_1 \geq 1$,
we then increase $A$ from~\eqref{holder-choix1-A} in order that
\begin{eqnarray*}
\frac{1}{2} A C(\nu,\tau)\psi'(r)r^{1-\beta} = \frac{1}{2} A C(\nu,\tau)C_2 \tau r^{\tau - \beta} e^{-C_2r^\tau} \geq 2(K + C_H) \geq 2K \psi(r) + \frac{2C_H}{C_1},
\end{eqnarray*}
which leads to the choice
\begin{eqnarray*}\label{holder-choix2-A}
A\geq \mathop{\rm max} \left\{\frac{2C_H}{C_f r_0} \, , \,  
\frac{4(K+C_H)r_0^{\beta-\tau}e^{C_2r_0^\tau}}{C(\nu,\tau) C_2\tau} \right\}.
\end{eqnarray*}
Therefore, inequality \eqref{holder-ineg-contra} now becomes
\begin{eqnarray}\label{holder-final}
&&\lambda \sqrt{\delta} + \frac{1}{2}\Phi(x,y)C(\nu,\tau)\psi'(r)(1+ C_2r^\tau) r^{1-\beta} \leq C(H,\nu,\mu)\psi'(r)\Phi(x,y) + 2C_1 \delta K.
\end{eqnarray}

The rest of the proof is only to fix parameters in order to reach a contradiction in~\eqref{holder-final}. It is now played with the main role of the ellipticity.

Recalling that $\beta > 1$, we fix 
$$r_s = \min \left\{\left(\frac{C(\nu, \tau)}{2C(H,\nu,\mu)}\right)^{\frac{1}{\beta -1}};r_0\right\}; \qquad C_2 = \frac{2C(H,\nu,\mu)}{C(\nu,\tau)}\max \lbrace r_s^{\beta - 1 - \tau};r_0^{\beta - 1 - \tau}\rbrace.$$
If $r\leq r_s$ then from the choice of $r_s$, we have
\begin{eqnarray*}
&& \frac{1}{2}C(\nu,\tau)(1+ C_2r^\tau) r^{1-\beta} \geq \frac{1}{2}C(\nu,\tau)r^{1-\beta} \geq \frac{1}{2}C(\nu,\tau)r_s^{1-\beta} \geq C(H,\nu,\mu).
\end{eqnarray*}
If $r_s \leq r \leq r_0$, we consider two cases
\\
$\bullet$ If $1+\tau - \beta \geq 0$, because of the choice of $C_2$ from above we have
\begin{eqnarray*}
&& \frac{1}{2}C(\nu,\tau)(1+ C_2 r^\tau) r^{1-\beta} \geq \frac{1}{2}C(\nu,\tau)C_2 r^{1+\tau - \beta} \geq \frac{1}{2}C(\nu,\tau)C_2 r_s^{1+\tau - \beta} \geq C(H,\nu,\mu).
\end{eqnarray*}
$\bullet$ If $1+\tau - \beta \leq 0$, because of the choice of $C_2$ from above we have
\begin{eqnarray*}
&& \frac{1}{2}C(\nu,\tau)(1+ C_2 r^\tau) r^{1-\beta} \geq \frac{1}{2}C(\nu,\tau)C_2   r^{1+\tau - \beta} \geq \frac{1}{2}C(\nu,\tau)C_2  r_0^{1+\tau - \beta} \geq C(H,\nu,\mu).
\end{eqnarray*}
Therefore, in any case, due to the choice of the constants, inequality \eqref{holder-ineg-contra} reduces to
$$\lambda \sqrt{\delta} \leq 2C_1 \delta K.$$
This is not possible for $\delta$ small enough. Then we get a contradiction in \eqref{global-new}.
\smallskip

\noindent{\it 5. Conclusion.} For every $x,y \in \R^N$, we have proved that
$$|u(x) - u(y)| \leq \sqrt{\delta} + (\psi(|x-y|)+\delta)C_1(\phi(x) + \phi(y) + A).$$
Since $A, \psi$ do not depend on $\delta$, letting $\delta \to 0$ and recalling that $\psi(r) = 1 - e^{-C_2r^\tau} \leq C_2 r^\tau$, for $\tau \in (0,1)$. Hence we get
$$|u(x) - u(y)| \leq C_1C_2|x-y|^\tau(\phi(x) + \phi(y) + A) = C_1C_2(1+A)|x-y|^\tau(\phi(x) + \phi(y))$$
and finally, \eqref{holder-estim-lip-u2} holds with $C= C_1C_2(1+A)$.
\end{proof}

\subsection{Regularity of solutions for degenerate equations.}

In this section, the equation \eqref{aperg} is degenerate which means that
\eqref{ellipticite1} does not necessarily hold (for the local case) and
$\beta\in(0,1]$ (for the nonlocal one). In these cases, we need to strengthen the assumption on $H$ by assuming that~\eqref{hypH2} holds and require a condition on the strength of the Ornstein-Uhlenbeck operator. 

\begin{Thm}  \label{lip-er-deg}
Let $\mu >0$, $u^{\lambda} \in C(\R^N)\cap \mathcal{E}_\mu(\R^N)$ be a solution of \eqref{aperg}.  Assume that~\eqref{dissipative},~\eqref{fil-f-term} and \eqref{hypH2} hold. If one of the followings holds

\noindent{(i) $\mathcal{F}(x,[u^\lambda]) = {\rm tr}(A(x) D^2u^\lambda (x))$
and \eqref{hyp-sig} holds.}

\noindent{(ii) $\mathcal{F}(x,[u^\lambda]) =
    \mathcal{I}(x,u^\lambda,Du^\lambda)$ and \eqref{M1}, \eqref{M2-assumpt} hold
    with $\beta \in (0,1]$.}
\\
Then there exists 
\begin{eqnarray*}
C(\mathcal{F},H)=
\begin{cases}
C(\sigma,H) ~ \text{in the local case }(i)\\
C(H) ~ \text{in the nonlocal case }(ii)
\end{cases}
\end{eqnarray*}
such that, for any $\alpha > C(\mathcal{F},H)$, there exists a constant
$C$ independent of $\lambda$ such that 
\begin{eqnarray}\label{estim-lip-deg}
&& |u^\lambda(x)-u^\lambda(y)|\leq C|x-y|(\phi_\mu(x)+\phi_\mu(y)), \quad x,y\in\R^N, \ \lambda \in (0,1).
\end{eqnarray}
\end{Thm}

\begin{proof}[\bf Proof of Theorem \ref{lip-er-deg}]
The beginning of the proof follows line to line from Step $1$ to Step $6$ excepting Step~$5$ (estimate for $H$-terms) as in the proof of Theorem \ref{lip-er}.

The viscosity inequality that we have to estimate in order to get a contradiction is
\begin{eqnarray}\label{ineq-visco-deg}
&& \lambda (u(x)-u(y)) -(\mathcal{F}(x,[u])-\mathcal{F}(y,[u])) \\ 
&& +\langle b(x),D_x\varphi\rangle - \langle b(y),-D_y\varphi\rangle + H(x,D_x\varphi) - H(y,-D_y\varphi)\nonumber\\
& \leq& f(x)-f(y),\nonumber
\end{eqnarray}
where $\varphi$ is defined in \eqref{varphi}, $D_x\varphi$ and $D_y\varphi$ are given by \eqref{Dxphi}, $\mathcal{F}(x,[u]) = {\rm tr}(A(x)X)$ and $\mathcal{F}(y,[u]) = {\rm tr}(A(y)Y)$ in the local case and $\mathcal{F}(x,[u]) = \mathcal{I}(x,u,D_x \varphi)$ and $\mathcal{F}(y,[u]) = \mathcal{I}(y,u,-D_y \varphi)$ in the nonlocal one.

Since we are doing the proof for degenerate equation, we do not need to construct very complicated concave test functions as we built in the one of Theorem \ref{lip-er} in order to get the ellipticity. We only need to take
\begin{eqnarray}\label{deg-psi}
\psi(r) = r, ~~ \forall r \in [0,\infty)
\end{eqnarray}
 in the both local and nonlocal case.

Now using \eqref{hypH2} and \eqref{Dxphi} to estimate for $H$-terms in \eqref{ineq-visco-deg},  we have
\begin{eqnarray}\label{estimH1}
&& H(x,D_x\varphi) - H(y,-D_y\varphi)\\
 &=& H(x,D_x\varphi) - H(x,-D_y\varphi) +  H(x,-D_y\varphi)- H(y,-D_y\varphi)\nonumber\\  
&\geq& -C_1L_H(\psi(|x-y|)+\delta)(|D\phi(x)| + |D\phi(y)|)(1+|x| + |x-y|)\nonumber\\
& & -L_{1H} |x-y| 
 - L_{1H}|x-y|\psi'(|x-y|) \Phi(x,y),\nonumber
\end{eqnarray}
here $L_H = \max\{L_{1H}, L_{2H}\}$.
\smallskip

\noindent{\it 1. Proof of $(i).$} We first use Lemma \ref{trace} \eqref{trace-deg} to estimate the difference of two local terms in \eqref{ineq-visco-deg}. Then plugging~\eqref{term-en-lambda}, \eqref{estim-b}, \eqref{estimf}, \eqref{trace-deg} and  \eqref{estimH1} into \eqref{ineq-visco-deg}, letting $\zeta \to 0$ and using~\eqref{deg-psi}, we obtain
\begin{eqnarray}\label{ineg-contra-deg}
&& \lambda \sqrt{\delta} + C_1\alpha  (\phi(x)+\phi(y)+A)  |x-y|\\\nonumber
&\leq& C_1 (\mathcal{C}_{\sigma} + L_{1H})(\phi(x)+\phi(y)+A)|x-y|  +L_{1H}|x-y| + C_f(\phi(x)+\phi(y))|x-y|\\\nonumber
&& - C_1(|x-y|+\delta)\mathcal{L}_{L}[\phi](x,y),\nonumber
\end{eqnarray}
where $\mathcal{L}_L$ is introduced in \eqref{phi} in the local case with $L(x,y) = 2(\mathcal{C}_\sigma + L_H)(1+|x|+|y|)$. Taking $\alpha > C(\sigma,H):=4(\mathcal{C}_\sigma + L_H)$, applying Lemma \ref{Lemphi} \eqref{etm-phi} and since $L_{1H} \leq L_H$, inequality~\eqref{ineg-contra-deg} reduces to
\begin{eqnarray*}
&& \lambda \sqrt{\delta} + C_1\alpha  (\phi(x)+\phi(y)+A)  |x-y|\\\nonumber
&\leq& C_1 (\mathcal{C}_{\sigma} + L_{H})(\phi(x)+\phi(y)+A)|x-y|\\\nonumber
&& + C_1(|x-y|+\delta)2K +L_{H}|x-y| + C_f(\phi(x)+\phi(y))|x-y|.
\end{eqnarray*}
Now we fix 
\begin{eqnarray*}
C_1 \geq \frac{4C_f}{3\alpha} + 1; \qquad A\geq \frac{4}{3\alpha}(K+L_H).
\end{eqnarray*}
By these choices and noticing that $\psi' = 1$, $C_1 \geq 1$, we obtain
\begin{eqnarray}\label{degf}
&&\frac{3}{4}C_1(\phi(x) + \phi(y) + A) \alpha |x-y| \\
&=& \frac{3}{4}C_1 (\phi(x) + \phi(y)) \alpha |x-y| + \frac{3}{4}C_1 A \alpha |x-y| \nonumber\\
&\geq& C_f(\phi(x) + \phi(y))|x-y| + C_1 K|x-y| + L_H|x-y|.\nonumber
\end{eqnarray}
Taking into account \eqref{degf} and noticing that $\alpha > 4(\mathcal{C}_\sigma + L_H)$, inequality \eqref{ineg-contra-deg} now becomes
$$\lambda \sqrt{\delta} \leq 2C_1K \delta.$$
This is not possible for $\delta$ small enough, hence we reach a contradiction.
\smallskip

\noindent{\it 2. Proof of (ii).} Using Proposition \ref{nlocalgen} \eqref{degnon} to estimate the difference of two nonlocal terms in \eqref{ineq-visco-deg}, then plugging~\eqref{term-en-lambda}, \eqref{estim-b}, \eqref{estimf}, \eqref{degnon} and \eqref{estimH1}  
into \eqref{ineq-visco-deg}, using \eqref{deg-psi} we obtain
\begin{eqnarray*}\label{ineg-contra-deg1}
&& \lambda \sqrt{\delta} + C_1\alpha  (\phi(x)+\phi(y)+A)  |x-y|\\\nonumber
&\leq& C_1 L_{1H}(\phi(x)+\phi(y)+A)|x-y|  +C_{H}|x-y| + C_f(\phi(x)+\phi(y))|x-y|\\\nonumber
&& - C_1(|x-y|+\delta)\mathcal{L}_{L}[\phi](x,y),\nonumber
\end{eqnarray*}
where $\mathcal{L}_L$ is introduced in \eqref{phi} with $\mathcal{F}$ is the nonlocal operator defined by \eqref{F-nonlocal} and $L(x,y)=2L_H(1+|x|+|y|).$ Taking $\alpha > 4L_H $ and using the same arguments as in the local case  $(i)$  we get also a contradiction.
\smallskip

\noindent{\it 3. Conclusion.}
 For any $\alpha > C(\mathcal{F},H)=
\begin{cases}
C(\sigma,H) ~ \text{in the local case }(i)\\
C(H) ~ \text{in the nonlocal case }(ii)
\end{cases},$
we have proved that $M\leq 0$ in both the local and the nonlocal case.
It follows that, for every $x,y\in\R^N,$
\begin{eqnarray*}
|u(x)-u(y)|\leq \sqrt{\delta}+(\psi(|x-y|)+\delta)C_1(\phi(x)+\phi(y)+A).
\end{eqnarray*}
Since $A,\psi$ do not depend on $\delta,$ letting $\delta\to 0$
we get
\begin{eqnarray*}
&& |u(x)-u(y)|\leq C_1 |x-y|(\phi(x)+\phi(y)+A)
= C_1(A+1)|x-y|(\phi(x)+\phi(y)).
\end{eqnarray*}
Here $\phi$ stands for $\phi_\mu$ and since $\phi\geq 1,$ \eqref{estim-lip-deg} holds with
$C=C_1(A+1).$
\end{proof}

\begin{Remark}\label{rm-deg-1}
 If $\sigma$ is a constant matrix, i.e., $L_\sigma = 0$ in \eqref{hyp-sig} then \eqref{trace-deg} reduces to
\begin{eqnarray*}\label{ifl-trace}
 && -{\rm tr}(A(x)X - A(y)Y) 
\geq  -C_1(\psi + \delta) ({\rm tr}(A(x)D^2\phi(x)) + {\rm tr}(A(y)D^2\phi(y)))+O(\zeta).
\end{eqnarray*}

Therefore using similarly arguments with the proof of the theorem we can prove that~\eqref{estim-lip-deg} holds for any $\alpha > C(H)$ (constant depends only on $H$).
\end{Remark}

\begin{Remark} \label{rm-deg}
The natural extension (regarding regularity with respect to $x$ of $H$) is
\begin{eqnarray}
\begin{cases}
\label{hypH4}
\text{ There exist } L_{1H},L_{2H} >0 \text{ such that for all } x,y,p,q\in\R^N \\
 |H(x,p)-H(y,p)|\leq L_{1H}|x-y|(1+|p|), \\
 |H(x,p)-H(x,q)| \leq L_{2H}|p-q|.
\end{cases}
 \end{eqnarray}
If $\sigma$ is a constant matrix, i.e., $L_\sigma = 0$ in \eqref{hyp-sig} and $H(x,p) = H(p)$ is Lipschitz continuous ($L_{1H} = 0$ in \eqref{hypH4}) (as in \cite{fil06}), then $C(\sigma,H) = 0$ in the proof, meaning that Theorem~\ref{lip-er-deg} $(i)$ holds with further assumption on the strength of the Ornstein-Uhlenbeck operator. It means that \eqref{estim-lip-deg} holds for any $\alpha > 0$ and this conclusion is still true in the nonlocal case.
\end{Remark}

\begin{Remark}\label{rm-deg1}
If $\sigma$ is any bounded lipschitz continuous symmetric matrix, i.e. \eqref{hyp-sig} holds. Assume \eqref{hypH4} with $L_{1H} = 0$, then the conclusion of Theorem \ref{lip-er-deg} should be there exists $C(\mathcal{F})=C(\sigma)$ (in the local case only) such that for any $\alpha > C(\mathcal{F})$, \eqref{estim-lip-deg} holds. The condition on $\alpha$ here is to compensate the degeneracy only.
\end{Remark}


\section{Regularity of the solutions for the evolution equation}\label{sec-evo}
In this section we establish Lipschitz estimates  (in space) for the solutions of the Cauchy problem \eqref{cauchy},
which are uniform in time.
\subsection{Regularity in the uniformly parabolic case}
\begin{Thm}\label{lip-cauchy}
Let $u \in C(Q_T)\cap  \mathcal{E}_\mu(Q_T) $ be a solution of \eqref{cauchy}. In addition to the hypotheses of Theorem \ref{lip-er}, assume that $u_0$ satisfies \eqref{fil-f-term} with a constant $C_0$. Then there exists a constant $C > 0$ independent of $T$ such that
\begin{eqnarray}\label{lipcon-u}
&& |u(x,t) - u(y,t)| \leq C|x-y|(\phi_\mu(x) + \phi_\mu(y)) \text{ for } x,y \in \R^N, \ t \in [0,T).
\end{eqnarray}
\end{Thm}

\begin{Remark}
If we have comparison theorem for \eqref{cauchy}, then classical techniques allow to deduce from Theorem \ref{lip-cauchy},  Lipschitz estimates in time for the solution, see \cite[Theorem~3.2]{fil06}. More generally, for such kind of equations, Lipschitz estimates in space imply H\"older estimates in time, see \cite[Lemma 9.1]{bbl02} for instance.
\end{Remark}

\begin{proof}[\bf Proof of Theorem \ref{lip-cauchy}]

We only give a sketch of proof since it is close to the proof of Theorem \ref{lip-er}.

Let $\epsilon ,\delta, A ,C_1>0$ and a $C^2$ concave and increasing function
$\psi:\R_+ \rightarrow \R_+$ with $\psi(0)=0$ which is defined as in the proof of Theorem \ref{lip-er} depending on the local or nonlocal case such that 
\begin{eqnarray}\label{c1c0}
C_1 \psi'_{\rm min} \geq C_0.
\end{eqnarray}

We consider
\begin{eqnarray*}
&& M =\sup_{(\R^N)^2 \times [0,T)}
\left\{
u(x,t)-u(y,t)-(\psi(|x-y|)+\delta)\Phi(x,y) -\frac{\epsilon}{T-t}
\right\},
\end{eqnarray*}
where $\Phi(x,y)= C_1(\phi(x)+\phi(y)+A)$, $\phi = \phi_\mu$ defined by \eqref{exp-growth}. Set $\varphi(x,y,t) =(\psi(|x-y|)+\delta)\Phi(x,y) -\frac{\epsilon}{T-t} $.

If $M \leq 0$ for some good choice of $A,C_1$ $\psi$ independent of $\delta, \epsilon >0$, then we get some locally uniform estimates by letting
$\delta, \epsilon \to 0$ . So we argue by contradiction,
assuming that $M >0.$

Since $u\in \mathcal{E}_\mu(\R^N)\cap C(Q_T) ,$ the supremum is achieved
at some point $(x,y,t)$ with $x \neq y$ thanks to $\delta, \epsilon >0$ and the continuity of $u$.

Since $M>0$, if $t=0$, from \eqref{fil-f-term} and by concavity of $\psi$, i.e., $\psi(|x-y|) \geq \psi'(|x-y|)|x-y|\geq \psi'_{\rm min} |x-y|$, we have
\begin{eqnarray*}
0 < M &=& u(x,0)-u(y,0)-C_1(\psi(|x-y|)+\delta)(\phi(x)+\phi(y)+A)-\frac{\epsilon}{T}\\
&\leq& C_0|x-y|(\phi(x)+\phi(y)+A) - C_1\psi'_{\rm min} |x-y|(\phi(x)+\phi(y)+A)-\frac{\epsilon}{T}.
\end{eqnarray*}
The last inequality is strictly negative due to \eqref{c1c0}. Therefore $t>0$. Now we can apply \cite[Theorem 8.3]{cil92} in the local case and \cite[Corollary 2]{bi08} in the nonlocal one to learn that, for any $\varrho > 0$, there exist $a,b \in \R$ and $X,Y \in \mathcal{S}^N$ such that
$$(a,D_x\varphi, X) \in \overline{\mathcal{P}}^{2,+}u(x,t); \qquad (b,-D_y\varphi, Y) \in \overline{\mathcal{P}}^{2,-}u(y,t),$$
\begin{eqnarray*}
a-b = \varphi_t(x,y,t) = \frac{\epsilon}{(T-t)^2} \geq \frac{\epsilon}{T^2}, \text{ and }
\left(\begin{array}{cc}
X &O\\
O&-Y
\end{array}
\right)
\leq 
A + \varrho A^2,
\end{eqnarray*}
where $A = D^2\varphi(x,y,t)$ 
and $\varrho A^2 = O(\varrho)$ ($\varrho$ will be sent to $0$ first).

Writing and subtracting the viscosity inequalities at $(x,y,t)$ in the local and nonlocal case we have
\begin{eqnarray}\label{final1}
&& \frac{\epsilon}{T^2} - (\mathcal{F}(x,[u]) - \mathcal{F}(y,[u])) +\langle b(x), D_x\varphi\rangle - \langle b(y), -D_y\varphi\rangle \\
&& + H(x, D_x\varphi)- H(y, -D_y\varphi)\nonumber \\
\nonumber
&\leq& f(x) - f(y),\nonumber
\end{eqnarray}
where $\mathcal{F}(x,[u]) = {\rm tr}(A(x)X)$ and $\mathcal{F}(y,[u]) = {\rm tr}(A(y)Y)$ in the local case and $\mathcal{F}(x,[u]) = \mathcal{I}(x,u,D_x \varphi)$ and $\mathcal{F}(y,[u]) = \mathcal{I}(y,u,-D_y \varphi)$ in the nonlocal one.

All of the different terms in \eqref{final1} are estimated as in the proof of Theorem \ref{lip-er}. We only need to fix
\begin{eqnarray*}\label{epsidel1}
\delta = \epsilon \min \lbrace 1, \frac{1}{3C_1KT^2}\rbrace,~~\text{then } \frac{\epsilon}{T^2}> 2C_1\delta K,
\end{eqnarray*}
where $K>0$ is a constant coming from \eqref{estim-phi2}. Therefore we reach a contradiction as in the proof of Theorem \ref{lip-er}.
%

We have proved that $M \leq 0$ for $\delta$ small enough.
It follows that, for every $x,y\in \R^N,$ $t\in [0,T)$,
\begin{eqnarray*}
|u(x,t)-u(y,t)|\leq C_1(\psi(|x-y|)+\delta)(\phi(x)+\phi(y)+A) +\frac{\epsilon}{T-t}.
\end{eqnarray*}
Since ${A},C_1, \psi$ do not depend on $\delta, \epsilon,$ letting $\delta, \epsilon \to 0$
and using the fact that $\psi$ is a concave increasing function, we have $\psi(r) \leq \psi'(0)r $. Finally, for all $x,y\in \R^N,$ $t\in [0,T)$ we get
\begin{eqnarray*}
|u(x,t)-u(y,t)|\leq C_1\psi'(0)|x-y|(\phi(x)+\phi(y)+A)
\leq C_1\psi'(0)(A+1)|x-y|(\phi(x)+\phi(y)).
\end{eqnarray*}
Since $\phi\geq 1,$ \eqref{lipcon-u} holds with
$C= C_1\psi'(0)(A+1).$
\end{proof}

\subsection{Regularity of solutions for degenerate parabolic equation}
\begin{Thm}\label{lip-cauchy-deg}
  Let $u \in C(Q_T)\cap  \mathcal{E}_\mu(Q_T) $ be a solution of \eqref{cauchy}. In addition to the hypotheses of Theorem \ref{lip-er-deg}, assume that $u_0$ satisfies \eqref{fil-f-term}. There exists $C(\mathcal{F},H)$ as in Theorem~\ref{lip-er-deg},
  such that, for any $\alpha > C(\mathcal{F},H)$,~\eqref{lipcon-u} holds for some $C>0$ independent of $T.$
\end{Thm}

The proof of this Theorem is an adaptation of the one of Theorem \ref{lip-er-deg} using the same extension to the parabolic case as explained in the proof of Theorem \ref{lip-cauchy}. So we omit here.

We have also the same Remarks as presented for the elliptic equations.

\section{Estimates for the growth function, local and nonlocal operators.}\label{sec-ests}
\subsection{Estimates for exponential growth function}
\begin{proof}[\textbf{Proof of Lemma \ref{Lemphi}}]\text{ }
\smallskip

\noindent{\it 1. Estimates on $\phi:$}
Recalling that $\phi(x)=e^{\mu \sqrt{|x|^2+1}}$ and setting $\langle x \rangle = \sqrt{|x|^2+1}$, for $x \in \R^N $, we have
\begin{eqnarray}\label{d1d2phi}
D\phi=\frac{\mu x}{\langle x \rangle}\phi (x), \quad D^2\phi=\frac{\mu \phi (x) }{\langle x \rangle}\Big[ I -(1-\mu \langle x \rangle) \frac{x}{\langle x \rangle}\otimes \frac{x}{\langle x \rangle}\Big].
\end{eqnarray}
\noindent{\it 2. Estimates on the Ornstein-Uhlenbeck operator.}
From \eqref{dissipative} and \eqref{d1d2phi} we have 
\begin{eqnarray}\label{lemphi-b}
\langle b(x),D\phi(x)\rangle &=& \langle b(x)-b(0),D\phi(x)\rangle + \langle b(0),D\phi(x)\rangle\\
& = & \langle b(x)-b(0),\mu \frac{x}{\langle x \rangle}\phi(x)\rangle + \langle b(0),\mu \frac{x}{\langle x \rangle}\phi(x)\rangle\nonumber\\
 &\geq & \mu\phi(x) \left(\frac{\alpha |x|^2}{\langle x \rangle}  -|b(0)|\right).\nonumber
\end{eqnarray}
\noindent{\it 3. Estimates for the local term.}
From \eqref{d1d2phi}, we compute that 
\begin{eqnarray}\label{lemphi-tr}
{\rm tr}(A(x)D^2\phi(x)) &=& \frac{\mu \phi (x) }{\langle x \rangle}\Big[ {\rm tr}(A(x)) - (1-\mu \langle x \rangle) {\rm tr} \Big(A(x) \frac{x}{\langle x \rangle}\otimes \frac{x}{\langle x \rangle}\Big)\Big] \\
&\leq& \frac{\mu \phi (x) }{\langle x \rangle} \Big[ | \sigma|^2 + \mu \langle x \rangle   |\sigma|^2 \Big] \nonumber \\
&\leq& C(\mu,| \sigma |)\phi(x).\nonumber
\end{eqnarray}
\noindent{\it 4. Estimates for the nonlocal term.}
We have
\begin{eqnarray}\label{Iphi}
&&\mathcal{I}(x,\phi,D\phi)\\ 
&=&  \int_B (\phi(x+z) - \phi(x) - \langle D\phi(x), z\rangle)\nu(dz) + \int_{B^c}(\phi(x+z) - \phi(x))\nu(dz)\nonumber\\
&=& \int_B \int_0^1(1-s)\langle D^2\phi(x+sz)z,z\rangle ds \nu(dz) + \phi(x) \int_{B^c}\left(\frac{\phi(x+z)}{\phi(x)} - 1 \right)\nu(dz)\nonumber\\
&\leq& \int_B \int_0^1 | D^2\phi(x+sz)||z|^2ds \nu(dz) + \phi(x) \int_{B^c}\left(\frac{\phi(x+z)}{\phi(x)} - 1 \right)\nu(dz)\nonumber
\end{eqnarray}
From \eqref{d1d2phi} we have
\begin{eqnarray}\label{phi2}
&& | D^2\phi(x)| \leq \frac{\mu \phi (x) }{\langle x \rangle}\Big[|I| + \Big|  \frac{x}{\langle x \rangle}\otimes \frac{x}{\langle x \rangle} \Big| + \mu \langle x \rangle\Big| \frac{x}{\langle x \rangle}\otimes \frac{x}{\langle x \rangle}\Big|\Big]\leq C(\mu, N)\phi(x).
\end{eqnarray}
On the other hand, for any $a,b$ we have
\begin{eqnarray*}
1 + (a+b)^2 \leq 1+ a^2 + 2\sqrt{1+a^2}\sqrt{1+b^2} + b^2  \leq \left( \sqrt{1+a^2} + \sqrt{1+b^2}\right)^2.
\end{eqnarray*}
This implies $\sqrt{1+(a + b)^2} \leq \sqrt{1+a^2} + \sqrt{1+b^2}$ and therefore
\begin{eqnarray}\label{phi-product}
\phi(x+z) = e^{\mu\sqrt{1+|x+z|^2}} \leq \phi(x)\phi(z), ~~ \forall x,z \in \R^N.
\end{eqnarray}

Using \eqref{phi2} and \eqref{phi-product} we have, for all $s \in [0,1]$,
$
|D^2\phi(x+sz)|
\leq C(\mu, N) \phi(x) \phi(z).
$
Hence, inequality \eqref{Iphi} becomes
\begin{eqnarray*}
\mathcal{I}(x,\phi,D\phi) &\leq& C(\mu, N) \phi(x)\int_B \phi(z)|z|^2 \nu(dz) + \phi(x) \int_{B^c} (\phi(z)-1)\nu(dz).
\end{eqnarray*}

Then, using $\eqref{M1}$, we get
\begin{eqnarray}\label{lemphi-i}
\mathcal{I}(x,\phi,D\phi) \leq C(\mu, N)C^1_\nu \phi(x) + C^1_\nu\phi(x) = C(\mu,\nu)\phi(x).
\end{eqnarray}
\noindent{\it 5. Estimate of the Lemma and end the computations for both cases.}
Let $L_0>0$, $L(x,y):= L_0(1+|x| + |y|)$ and set 
$$\mathcal{L}_L[\phi](x,y):= - \mathcal{F}(x,[\phi]) - \mathcal{F}(y,[\phi]) +\langle b(x),D\phi(x)\rangle +\langle b(y),D\phi(y)\rangle - L(x,y)(|D\phi(x)| + |D\phi(y)|).$$

Set
\begin{eqnarray*}
C(\mathcal{F}) = 
\begin{cases}
C(\mu, | \sigma |),& \text{if } \mathcal{F} \text{ is defined by } \eqref{F-diffusion}\\
C(\mu,\nu), & \text{if } \mathcal{F} \text{ is defined by } \eqref{F-nonlocal}.
\end{cases}
\end{eqnarray*}
Since $|D\phi(x)| \leq \mu \phi(x)$, from \eqref{lemphi-b}, \eqref{lemphi-tr} and \eqref{lemphi-i} we have
\begin{eqnarray}\label{estimate-phi}
&&\mathcal{L}_L[\phi](x,y) \geq
\mu \phi(x)\left( a(x) - L_0|x| - L_0|y|\right) + \mu \phi(y)\left( a(y) - L_0|x| - L_0|y|\right),\nonumber
\end{eqnarray}
where $a(x) =  \alpha \frac{|x|^2}{\langle x\rangle}- |b(0)|- C(\mathcal{F}) -  L_0,$ $\forall x \in \R^N.$ Now we define
$$R_x:= \frac{2}{\alpha - 2L}\left( \frac{2}{\mu} + |b(0)|+C(\mathcal{F}) +  L_0 + L_0|x|\right) $$
and take $\alpha > 2L_0$, if $|x|\geq R_y \text{ and } |y|\geq R_x,$
then $\mathcal{L}_L [\phi](x,y) \geq 2\phi(x) + 2 \phi(y).$ Set
\begin{eqnarray*}
\mathop{\rm sup}_{y\in{B}(0,R_x)} \{\mu \phi(y)\left( -a(y) + L_0|x| + L_0|y|\right)\} =:K_x,
\end{eqnarray*}
then we get that, for all $x,y\in\R^N,$ $\mathcal{L}_L[\phi](x,y) \geq \phi(x) + \phi(y) +\phi(x) - K_x + \phi(y) - K_y.$
Since $\alpha > 2L_0$, we have $\sup_{\R^N}\{-\phi(x) + K_x\}, \sup_{\R^N}\{-\phi(y) + K_y\} < + \infty$. Hence, define
$K:= \sup \{ \sup_{\R^N}\{-\phi(x) + K_x\}, \sup_{\R^N}\{-\phi(y) + K_y\}\},$ we obtain
\begin{eqnarray}\label{lemphi-final}
\mathcal{L}_L[\phi](x,y) \geq \phi(x) + \phi(y) - 2K.
\end{eqnarray}

We now suppose that $L(x,y) = L_0$ independent of $x,y$. 
Setting
$$ R_\phi:= \frac{2}{\alpha }\left( \frac{1}{\mu} + |b(0)|+C(\mathcal{F}) +  L_0 \right), ~~K:=\mathop{\rm sup}_{x\in{B}(0,R_\phi)} \{\mu \phi(x)\left( -a(x) \right)\}$$
and using the same arguments as above we obtain that \eqref{lemphi-final} holds for any $\alpha >0$. Since $\phi(x)\to +\infty$ as $x\to +\infty,$ hence there exists $R\geq R_\phi$ such that
\begin{eqnarray*}
&& -\mathcal{F}(x,[\phi]) + \langle b(x),D\phi(x)\rangle -L_0|D\phi(x)|\geq \phi(x)-K\geq
\left\{
\begin{array}{ll}
 -K & \text{for $|x|\leq R,$}\\
K & \text{for $|x|\geq R.$}
\end{array}
\right. 
\end{eqnarray*}
Notice that $K$ and $R$ depend only on $\mathcal{F}, b, L, \mu $.
\end{proof}

\subsection{Estimates for the local operator}\label{sec-local}
\begin{proof}[\textbf{Proof of Lemma \ref{trace}}]\text{ }
\smallskip

\noindent{\it 1. Using the matrix inequality \eqref{ineg-matricielle}.} From \eqref{ineg-matricielle}, setting
\begin{eqnarray}\label{XxYy}
X_x= X- C_1(\psi+\delta)D^2\phi(x)
\quad {\rm and} \quad 
Y_y= Y+ C_1(\psi+\delta)D^2\phi(y).
\end{eqnarray}
We have, for every $\zeta, \xi\in\R^N,$
\begin{eqnarray*}
&&\langle X_x\zeta, \zeta\rangle - \langle Y_y\xi, \xi\rangle +O(\varrho)\\
&\leq&  
\psi'' \Phi \langle \zeta-\xi, p\otimes p(\zeta-\xi)\rangle
+ \psi' \Phi \langle \zeta-\xi, \mathcal{C}(\zeta-\xi)\rangle + C_1\psi'
[ 
 \langle p\otimes D\phi(x)+D\phi(x)\otimes p \zeta, \zeta\rangle
 \nonumber\\
& & +\langle  (p\otimes D\phi(y)-D\phi(x)\otimes p) \xi, \zeta\rangle + \langle (D\phi(y)\otimes p - p\otimes (D\phi(x))\zeta, \xi\rangle
\nonumber\\
& &-\langle (p\otimes D\phi(y)+D\phi(y)\otimes p) \xi, \xi\rangle ],\nonumber
\end{eqnarray*}
where $p$ and $\mathcal{C}$ are given by \eqref{p-C}. Set $\sigma_x = \sigma(x),$ $\sigma_y = \sigma(y)$.
\smallskip

\noindent{\it 2. Computing the trace with suitable orthonormal basis.}
Following Ishii-Lions~\cite{il90} and Barles~\cite{barles91b}, we choose an orthonormal  basis 
$(e_i)_{1\leq i\leq N}$ to compute
${\rm tr}(\sigma_x\sigma_x^TX)$ and another
one,  $(\tilde{e}_i)_{1\leq i\leq N}$ to compute
${\rm tr}(\sigma_x\sigma_x^TY).$
Now we estimate $T:={\rm tr} (A(x)X_x - A(y)Y_y) \ {\rm with}  \  A = \sigma\sigma^T$ in the following way:
\begin{eqnarray*}
T &=& \sum_{i=1}^N  \langle X_x \sigma_x e_i, \sigma_x e_i\rangle 
- \langle Y_y\sigma_y \tilde{e}_i, \sigma_y \tilde{e}_i\rangle
\\
&\leq & \sum_{i=1}^N
\psi'' \Phi   \langle p\otimes p Q_i,Q_i\rangle
+ \psi'  \Phi\langle \mathcal{C}Q_i,Q_i\rangle  + C_1\psi'[\langle p\otimes D\phi(x)\sigma_x e_i, Q_i\rangle \\
&&+ \langle D\phi(x)\otimes p Q_i, \sigma_x e_i \rangle   + \langle p\otimes D\phi(y)\sigma_y \tilde{e}_i , Q_i \rangle + \langle D\phi(y)\otimes p Q_i, \sigma_y \tilde{e}_i \rangle] + O(\varrho) \\
&\leq & \psi'' \Phi   \langle p,Q_1\rangle ^2 +
\sum_{i=1}^N  \psi'  \Phi\langle \mathcal{C}Q_i,Q_i\rangle  + C_1\psi'\mathcal{P}_i + O(\varrho).
\end{eqnarray*}
where we set $Q_i = \sigma_x e_i - \sigma_y \tilde{e}_i $, noticing that $\psi'' \Phi   \langle p\otimes p Q_i,Q_i\rangle = \psi'' \Phi   \langle p,Q_i\rangle ^2 \leq 0$ since $\psi$ is concave function and for all $1 \leq i \leq N,$
$$\mathcal{P}_i = \langle p\otimes D\phi(x)\sigma_x e_i, Q_i\rangle + \langle D\phi(x)\otimes p Q_i, \sigma_x e_i \rangle  + \langle p\otimes D\phi(y)\sigma_y \tilde{e}_i , Q_i \rangle + \langle D\phi(y)\otimes p Q_i, \sigma_y \tilde{e}_i \rangle.$$
We now set up suitable basis in two following cases.
\smallskip

\noindent{\it 2.1. Estimates for the trace when $\sigma$ is degenerate, i.e., \eqref{hyp-sig} holds only.} We choose any orthonormal basis such that $e_i = \tilde{e}_i$ . It follows
\begin{eqnarray*}
T  & \leq & \sum_{i=1}^N  \psi'  \Phi \langle \mathcal{C}(\sigma_x  - \sigma_y)e_i, (\sigma_x  - \sigma_y)e_i \rangle \\
& & + C_1\psi'[\langle p\otimes D\phi(x)\sigma_x e_i, (\sigma_x  - \sigma_y)e_i \rangle + \langle D\phi(x)\otimes p (\sigma_x  - \sigma_y)e_i, \sigma_x e_i \rangle  \\
& & + \langle p\otimes D\phi(y)\sigma_y e_i , (\sigma_x  - \sigma_y)e_i \rangle + \langle D\phi(y)\otimes p (\sigma_x  - \sigma_y)e_i, \sigma_y e_i \rangle] + O(\varrho)\\
& \leq & N\psi'  \Phi |\sigma_x  - \sigma_y|^2|\mathcal{C}| + NC_1\psi'[|p\otimes D\phi(x)||\sigma_x||\sigma_x  - \sigma_y| + |D\phi(x)\otimes p ||\sigma_x  - \sigma_y| |\sigma_x|\\
&& + |p\otimes D\phi(y)||\sigma_y||\sigma_x  - \sigma_y| + |D\phi(y)\otimes p ||\sigma_x  - \sigma_y| |\sigma_y|
] + O(\varrho).
\end{eqnarray*}
By \eqref{p-C} we first have $|\mathcal{C}| \leq 1/|x-y|$, and $|D\phi\otimes p |,|p\otimes D\phi| \leq  |D\phi|$. Using the fact that $\sigma$ is a Lipschitz and bounded function, i.e., \eqref{hyp-sig} holds, we obtain
\begin{eqnarray*}
T \leq N L_\sigma^2 |x-y|\psi' \Phi + 2NC_\sigma L_\sigma C_1(|D\phi(x)| + |D\phi(y)|)|x-y| \psi'. 
\end{eqnarray*}
Then by the concavity of $\psi$, i.e., $\psi'(|x-y|)|x-y| \leq \psi(|x-y|)$ and from \eqref{XxYy}, we get
\begin{eqnarray}\label{Tr1}
&&-{\rm tr}(A(x)X - A(y)Y) \\
&\geq&  - \tilde{C}_\sigma |x-y|\psi' \Phi -C_1(\psi + \delta) [{\rm tr}(A(x)D^2\phi(x)) + {\rm tr}(A(y)D^2\phi(y))\nonumber\\
&& \qquad  +\tilde{C}_\sigma (|D\phi(x)|+|D\phi(y)|)] +O(\varrho),\nonumber
\end{eqnarray}
where
\begin{eqnarray}\label{Ctilde}
\tilde{C}_\sigma  = \max \lbrace NL_\sigma^2, 2N C_\sigma L_\sigma \rbrace.
\end{eqnarray}

\noindent{\it 2.2. More precise estimate when $\sigma$ is strictly elliptic, i.e., \eqref{ellipticite1} holds.}
Since $\sigma$ is uniformly invertible, we can choose
$e_1=\frac{\sigma_x^{-1}p}{|\sigma_x^{-1}p|} \ , \tilde{e}_1 = - \frac{\sigma_y^{-1}p}{|\sigma_y^{-1}p|}.$ If $e_1$ and $\tilde{e_1}$ are collinear, then we can complete the basis with orthonormal unit vectors $e_i = \tilde{e_i} \in e_1^\perp, \ 2 \leq i \leq N.$ Otherwise, in the plane span $\{e_1, \tilde{e}_1\}$, we consider a rotation $\mathcal{R}$ of angle $\frac{\pi}{2}$ and define $e_2 = \mathcal{R}e_1,~~\tilde{e}_2 = -\mathcal{R}\tilde{e}_1.$ Finally, noticing that $\text{span}\{e_1,e_2\}^\perp = \text{span}\{\tilde{e}_1,\tilde{e}_2\}^\perp$, we can complete the orthonormal bases with unit vectors $e_i = \tilde{e}_i \in \text{span}\{e_1,e_2\}^\perp , \ 3 \leq i \leq N.$ 

For $i = 1$, we compute
\begin{eqnarray*}
&& Q_1 = \sigma_x e_1 - \sigma_y \tilde{e}_1 = [\frac{1}{|\sigma_x^{-1}p|} + \frac{1}{|\sigma_y^{-1}p|}]p ~~\Rightarrow ~~|\langle p,Q_1\rangle | = \frac{1}{|\sigma_x^{-1}p|} + \frac{1}{|\sigma_y^{-1}p|}.
\end{eqnarray*}
%
We have
\begin{eqnarray*}
&& |\sigma^{-1}p|^2 = \langle \sigma^{-1}p , \sigma^{-1}p \rangle = \langle (\sigma^{-1})^T \sigma^{-1}p , p \rangle
= \langle (\sigma\sigma^T)^{-1} p, p\rangle = \langle A^{-1} p , p\rangle \leq |A^{-1} p|.
\end{eqnarray*}
From \eqref{ellipticite1}, we have $\langle Ap , p \rangle \geq \rho |p|^2 \Rightarrow \langle AA^{-1} p , A^{-1} p \rangle\geq \rho|A^{-1}p|^2 \Rightarrow \langle p , A^{-1}p \rangle \geq \rho|A^{-1}p|^2.$ This implies $|A^{-1}p| \leq \frac{1}{\rho}$, so $|\langle p,Q_1\rangle | \geq 2 \sqrt{\rho}$ and that $|Q_1| - |\langle p,Q_1\rangle | = 0.$
Moreover,
\begin{eqnarray*}
&&\langle p\otimes D\phi(x)\sigma_x e_1, Q_1\rangle = \frac{1}{|\sigma_x^{-1}p|} [ \frac{1}{|\sigma_x^{-1}p|} +  \frac{1}{|\sigma_y^{-1}p|}]\langle D\phi(x) ,p \rangle\\
&&\langle D\phi(x)\otimes p Q_1, \sigma_x e_1 \rangle = \frac{1}{|\sigma_x^{-1}p|} [ \frac{1}{|\sigma_x^{-1}p|} +  \frac{1}{|\sigma_y^{-1}p|}]\langle D\phi(x) ,p \rangle
\\
&&\langle p\otimes D\phi(y)\sigma_y \tilde{e}_1 , Q_1 \rangle = - \frac{1}{|\sigma_y^{-1}p|} [ \frac{1}{|\sigma_x^{-1}p|} +  \frac{1}{|\sigma_y^{-1}p|}]\langle D\phi(y) ,p \rangle
\\
&&\langle D\phi(y)\otimes p Q_1, \sigma_y \tilde{e}_1 \rangle = - \frac{1}{|\sigma_y^{-1}p|} [ \frac{1}{|\sigma_x^{-1}p|} +  \frac{1}{|\sigma_y^{-1}p|}]\langle D\phi(y) ,p \rangle.
\end{eqnarray*}
Since 
$
C_\sigma|\sigma_x^{-1}p| \geq |\sigma_x||\sigma_x^{-1}p| =  1,$ we infer $ \frac{1}{|\sigma_x^{-1}p|} \leq C_\sigma.
$
Therefore,
$\mathcal{P}_1  \leq 4 C_\sigma^2 (|D\phi(x)|+|D\phi(y)|).
$ Next, using the concavity of $\psi$ and the above estimates, we obtain
\begin{eqnarray*}
&& T \leq 4\rho \psi'' \Phi + 4C_1C_\sigma^2\psi' (|D\phi(x)|+|D\phi(y)|) +  \sum_{i=2}^N
\frac{\psi'  \Phi}{|x-y|} |Q_i|^2  + C_1\psi'\mathcal{P}_i + O(\varrho)
\nonumber.
\end{eqnarray*}

For $i = 2$,
we compute
\begin{eqnarray}\label{Q2}
&& Q_2 = \sigma_x e_2 - \sigma_y \tilde{e}_2 = \sigma_x \mathcal{R} e_1 +\sigma_y \mathcal{R} \tilde{e}_1 = (\sigma_x - \sigma_y)\mathcal{R}e_1 + \sigma_y (\mathcal{R}e_1 + \mathcal{R} \tilde{e}_1).
\end{eqnarray}
$\bullet$ $|x-y| \leq 1$, we use the fact that $\sigma$ is lipschitz in \eqref{Q2} to obtain
\begin{eqnarray*}
|Q_2| \leq L_\sigma |x-y| + C_\sigma |\mathcal{R}e_1 + \mathcal{R} \tilde{e}_1| = L_\sigma |x-y| + C_\sigma |e_1 + \tilde{e}_1|.
\end{eqnarray*}
Moreover,
\begin{eqnarray*}
|e_1 + \tilde{e}_1| 
&\leq& \frac{1}{|\sigma_x^{-1}p|}|\sigma_x^{-1}p - \sigma_y^{-1}p| + |\sigma_y^{-1}p||\frac{1}{|\sigma_x^{-1}p|} - \frac{1}{|\sigma_y^{-1}p|}| \nonumber \\
&\leq& \frac{2}{|\sigma_x^{-1}p|}|\sigma_x^{-1} - \sigma_y^{-1}|
= \frac{2}{|\sigma_x^{-1}p|} |\sigma_y^{-1}[\sigma_y - \sigma_x]\sigma_x^{-1}|.
\end{eqnarray*}
Then,
\begin{eqnarray*}
&& |Q_2| \leq L_\sigma |x-y| + \frac{2{C_\sigma}L_\sigma |x-y|}{\sqrt{\rho}}=   (1+\frac{2{C_\sigma}}{\sqrt{\rho}} )L_\sigma|x-y|.
\end{eqnarray*}
$\bullet$ $|x-y| \geq 1$, by using the property that $\sigma$ is bounded in \eqref{Q2}, we get
\begin{eqnarray*}
|Q_2| \leq |\sigma_x| + |\sigma_y| + 2|\sigma_y| \leq 4C_\sigma.
\end{eqnarray*}
Therefore, for all $|x-y|>0,$ $|Q_2| \leq \max \lbrace (1+\frac{2{C_\sigma}}{\sqrt{\rho}} )L_\sigma,4C_\sigma\rbrace\min\{1,|x-y|\}.
$
On the other hand,
\begin{eqnarray*}
\mathcal{P}_2
&\leq&  |Q_2| [ |\sigma_x|(|p\otimes D\phi(x)| + |D\phi(x)\otimes p|) 
 + |\sigma_y|(|p\otimes D\phi(y)| + |D\phi(y)\otimes p|) ] \\
&\leq&  2C_\sigma L_\sigma(1+\frac{2{C_\sigma}}{\sqrt{\rho}} )|x-y|(|D\phi(x)|+|D\phi(y)|).
\end{eqnarray*}
Set $\hat{C}_\sigma : = \max \lbrace {\max}^2 \lbrace (1+\frac{2{C_\sigma}}{\sqrt{\rho}} )L_\sigma,4C_\sigma\rbrace,   2C_\sigma L_\sigma(1+\frac{2{C_\sigma}}{\sqrt{\rho}} ) \rbrace.
$
Hence, we obtain:
\begin{eqnarray*}
T &\leq& 4\rho \psi'' \Phi + 4C_1C_\sigma^2\psi' (|D\phi(x)|+|D\phi(y)|)  + \hat{C}_\sigma \psi'\Phi \\
&& +C_1\hat{C}_\sigma |x-y| \psi'(|D\phi(x)|+|D\phi(y)|)+  \sum_{i=3}^N
\frac{\psi'  \Phi}{|x-y|} |Q_i|^2 + C_1\psi'\mathcal{P}_i + O(\varrho)
\nonumber.
\end{eqnarray*}

Recall that $e_i = \tilde{e}_i  , \ 3 \leq i \leq N$, similarly with the above estimates we get
\begin{eqnarray*}
&& |Q_i| \leq {\max}\{L_{\sigma}, 2C_{\sigma}\}\min\{1, |x-y|\}~~
\text{ and }
 \mathcal{P}_i
\leq 2C_\sigma L_\sigma |x-y|(|D\phi(x)|+|D\phi(y)|).
\end{eqnarray*}

Finally, combining all the above estimates, we obtain
\begin{eqnarray}\label{Tr2}
T &\leq& 4 \rho \psi'' \Phi + 4C_1{C_\sigma}^2\psi'(|D\phi(x)|+|D\phi(y)|) + \hat{C}_\sigma \psi'\Phi  \\
& & + C_1\hat{C}_\sigma |x-y|\psi'(|D\phi(x)|+|D\phi(y)|) + (N-2){\max}^2\{L_{\sigma}, 2C_{\sigma}\}\psi' \Phi \nonumber \\
&& + (N-2)2C_1C_\sigma L_\sigma|x-y|\psi'(|D\phi(x)|+|D\phi(y)|)\nonumber\\
&\leq& 4 \rho \psi'' \Phi + 4C_1{C_{\sigma}^2}\psi'(|D\phi(x)|+|D\phi(y)|) + C_{\sigma1} \psi' \Phi \nonumber\\
&  & + C_1C_{\sigma2}|x-y|\psi' (|D\phi(x)|+|D\phi(y)|),\nonumber
\end{eqnarray}
where $C_{\sigma1} = \hat{C}_\sigma + (N-2){\max}^2\{L_{\sigma}, 2C_{\sigma}\};$  $C_{\sigma2} = \hat{C}_\sigma + (N-2)2C_\sigma L_\sigma.$
\smallskip

\noindent{\it 3. Conclusion.} Using the concavity of $\psi$ and choose
\begin{eqnarray}\label{Csigma}
&& \tilde{\tilde{C}}_{\sigma}(N,\rho,\sigma) := \max\{C_{\sigma1}, C_{\sigma2},4C_{\sigma}^2\},\qquad \mathcal{C}_\sigma = \max\{ \tilde{C}_\sigma, \tilde{\tilde{C}}_{\sigma} \},
\end{eqnarray}
where $\tilde{C}_\sigma$ is defined by \eqref{Ctilde}. Then from \eqref{Tr1} and \eqref{Tr2} we get the conclusion.
\end{proof}

\subsection{Estimates for the nonlocal operator}\label{sec-nonlocal}

\begin{proof}[\textbf{Proof of Proposition \ref{nlocalgen}.}]

Several parts of the proof are inspired by \cite{bcci12} and adapted to our unbounded framework. 
\smallskip

\noindent{\it 1. Proof of (i).} We split the domain of integration into two pieces, on the unit ball $B$ and its complement $B^c$. 

Let $(x,y)$ be a maximum point of $\Psi(\cdot,\cdot)$, we have 
\begin{eqnarray*}
 u(x+z) - u(y + z) -(u(x) - u(y ))\leq  C_1(\psi(|a|) + \delta)[\phi(x+z)-\phi(x) + \phi(y +z)  - \phi(y )].
\end{eqnarray*}
Taking the integral over $B^c$, we first get
\begin{eqnarray}\label{t1}
&&\mathcal{I}[B^c](x,u,D_x\varphi)-\mathcal{I}[B^c](y,u,-D_y\varphi) \\
&\leq& C_1(\psi(|a|) + \delta) \left(\mathcal{I}[B^c](x,\phi,D\phi) + \mathcal{I}[B^c](y,\phi,D\phi) \right).\nonumber
\end{eqnarray}

Moreover, at the maximum point, we have
\begin{eqnarray}\label{estu}
&& u(x + z) - u(x) - \langle D_x \varphi(x,y), z\rangle \\
&\leq&  u(y + z') - u(y)+ \langle D_y \varphi(x,y), z'\rangle + \varphi(x + z,y + z') - \varphi(x, y)\nonumber\\
& & + \langle D_y \varphi(x,y),z-z'\rangle - \langle D_x \varphi(x,y) + D_y \varphi(x,y), z\rangle\nonumber,
\end{eqnarray}
where $D_x\varphi$ and $D_y\varphi$ are given by \eqref{Dxphi}. Taking $z' = z$ in \eqref{estu} and using \eqref{Dxphi}, we have
\begin{eqnarray*}\label{estimatet2}
&&u(x+z) - u(x) - \langle D_x\varphi(x,y),z\rangle -(u(y+z) - u(y) +\langle D_y\varphi(x,y), z\rangle)\\
&\leq&  \varphi(x+z,y+z) - \varphi(x,y) -\langle D_x\varphi(x,y) + D_y\varphi(x,y), z\rangle\nonumber\\
&\leq& C_1(\psi(|a|) + \delta)[\phi(x+z) - \phi(x) - \langle D\phi(x), z\rangle + \phi(y+z) - \phi(y) - \langle D\phi(y), z \rangle].\nonumber
\end{eqnarray*}
Then, taking the integral over the ball we get
\begin{eqnarray}\label{t2}
&& \mathcal{I}[B](x,u,D_x\varphi)-\mathcal{I}[B](y,u,-D_y\varphi) 
\\
& \leq&  C_1(\psi(|a|) + \delta)  \left(\mathcal{I}[B ](x,\phi,D\phi) + \mathcal{I}[B](y,\phi,D\phi) \right).\nonumber
\end{eqnarray}

Therefore, from \eqref{t1} and \eqref{t2} we obtain
\begin{eqnarray*}
&& \mathcal{I}(x,u,D_x\varphi) - \mathcal{I}(y,u, -D_y\varphi) 
\leq C_1(\psi(|a|) + \delta) \{\mathcal{I}(x,\phi,D\phi)  + \mathcal{I}(y,\phi,D\phi)\}.
\end{eqnarray*}

\noindent{\it 2. Proof of (ii).} In this case, we split the domain of integration into three pieces,
$$
\mathcal{T}(x,y):=\mathcal{I}(x,u,D_x\varphi) - \mathcal{I}(y,u,-D_y\varphi) = \mathcal{T}^1(x,y) + \mathcal{T}^2(x,y) + \mathcal{T}^3(x,y),
$$
where $\mathcal{T}^1,$ $ \mathcal{T}^2,$ $\mathcal{T}^3$ are the difference of the nonlocal terms
over the domains $B^c,$ $B \setminus \mathcal{C}_{\eta, \gamma}(a),$ $\mathcal{C}_{\eta, \gamma}(a)$ respectively.

We argue as in the proof of (i) to first get
\begin{eqnarray}\label{t1t2}
&& \mathcal{T}^1(x,y) + \mathcal{T}^2(x,y) 
\leq  C_1(\psi(|a|) + \delta) \{\mathcal{I}[\mathcal{C}^c_{\eta, \gamma}(a)](x,\phi,D\phi)  + \mathcal{I}[\mathcal{C}^c_{\eta, \gamma}(a)](y,\phi,D\phi)\}.
\end{eqnarray}

Now the rest of the proof is only to estimate for $\mathcal{T}^3(x,y)$.
Taking $z' = 0$ and $z = 0$ in the inequality \eqref{estu} we get
\begin{eqnarray}\label{T3}
 \mathcal{T}^3(x,y) &\leq& \int_{\mathcal{C}_{\eta, \gamma}(a)} [ \varphi^1(x,y,z) + \varphi^2(x,y,z)] \nu(dz),
\end{eqnarray}
where
\begin{eqnarray*}
&& \varphi^1(x,y,z) = \varphi(x + z, y) - \varphi(x,y) -\langle D_x \varphi(x,y), z\rangle,\\
&&  \varphi^2(x,y,z) = \varphi(x , y+z) - \varphi(x,y) -\langle D_y \varphi(x,y), z\rangle.
\end{eqnarray*}
Let $\hat{a} = (x-y)/|x-y|.$ From \eqref{varphi} and \eqref{Dxphi} we have
\begin{eqnarray}\label{var1}
\varphi^1
&=& (\psi(|a+z|) + \delta)[ \Phi(x+z,y) - \Phi(x,y)] - C_1(\psi(|a|) + \delta)\langle D\phi(x), z\rangle\\
&& + [\psi(|a+z|)-\psi(|a|)- \psi'(|a|)\langle \hat{a}, z\rangle]\Phi(x,y)\nonumber\\
&=& C_1(\psi(|a+z|) - \psi(|a|))[ \phi(x+z) - \phi(x)] \nonumber\\
&& + C_1(\psi(|a|) + \delta)(\phi(x+z) - \phi(x)  -\langle D\phi(x), z\rangle )\nonumber\\
&& + [\psi(|a+z|)-\psi(|a|)- \psi'(|a|)\langle \hat{a}, z\rangle]\Phi(x,y).\nonumber
\end{eqnarray}
Similarly, we have
\begin{eqnarray}\label{var2}
\varphi^2
&=& C_1(\psi(|a-z|) - \psi(|a|))[ \phi(y+z) - \phi(y)] \\
&& + C_1(\psi(|a|) + \delta)(\phi(y+z) - \phi(y) -\langle D\phi(y), z\rangle )\nonumber\\
&& + [\psi(|a-z|)-\psi(|a|)+ \psi'(|a|)\langle \hat{a}, z\rangle]\Phi(x,y).\nonumber
\end{eqnarray}
Then from \eqref{T3}, \eqref{var1} and \eqref{var2}, we obtain
\begin{eqnarray}\label{T3-general}
 \mathcal{T}^3(x,y) &\leq&  C_1  \int_{\mathcal{C}_{\eta, \gamma}(a)} \{(\psi(|a+z|) - \psi(|a|))[ \phi(x+z) - \phi(x)]\\
 && \hspace*{2cm} + (\psi(|a-z|) - \psi(|a|))[ \phi(y+z) - \phi(y)]\}\nu(dz)\nonumber\\
 && + C_1(\psi(|a|) + \delta)(\mathcal{I}[\mathcal{C}_{\eta, \gamma}(a)](x,\phi,D\phi) + \mathcal{I}[\mathcal{C}_{\eta, \gamma}(a)](y,\phi,D\phi))\nonumber\\
 && + \Phi(x,y) \int_{\mathcal{C}_{\eta, \gamma}(a)} \{ \psi(|a+z|)-\psi(|a|)- \psi'(|a|)\langle \hat{a}, z\rangle 
\nonumber \\
 &&\hspace*{3cm} +  \psi(|a-z|)-\psi(|a|)+ \psi'(|a|)\langle \hat{a}, z\rangle \} \nu(dz).\nonumber
\end{eqnarray}
Because of the monotonicity and the concavity of $\psi$ we have
\begin{eqnarray}\label{T3psi}
\psi(|a+z|) - \psi(|a|) \leq \psi(|a| + |z|) - \psi(|a|) 
\leq \psi'(|a|)|z|.
\end{eqnarray}
Since $\phi\in C^\infty (\R^N)$ is a convex function and recalling \eqref{d1d2phi} and using \eqref{phi-product} we have 
\begin{eqnarray}\label{T3phix}
&&\phi(x+z) - \phi(x) \leq |D\phi(x+z)||z| \leq \mu \phi(x+z)|z| \leq  \mu \phi(x)\phi(z)|z|, ~\forall x,z \in \R^N.
\end{eqnarray}
Using \eqref{T3psi} and \eqref{T3phix} to estimate for \eqref{T3-general} we obtain
\begin{eqnarray*}
\mathcal{T}^3(x,y) & \leq  &C_1 \mu \psi'(|a|)(\phi(x) + \phi(y))\int_{\mathcal{C}_{\eta, \gamma}(a)} \phi(z)|z|^2 \nu(dz) \\
&&  + C_1(\psi(|a|) + \delta)(\mathcal{I}[\mathcal{C}_{\eta, \gamma}(a)](x,\phi,D\phi) + \mathcal{I}[\mathcal{C}_{\eta, \gamma}(a)](y,\phi,D\phi))\\
 && + \Phi(x,y) \int_{\mathcal{C}_{\eta, \gamma}(a)} [ \psi(|a+z|)-\psi(|a|)- \psi'(|a|)\langle \hat{a}, z\rangle 
 \\
 &&\qquad \qquad \qquad +  \psi(|a-z|)-\psi(|a|)+ \psi'(|a|)\langle \hat{a}, z\rangle ] \nu(dz).
\end{eqnarray*}
Thanks to Assumption~\eqref{M2-assumpt}, we can readily
apply \cite[Lemma 12]{bcci12} to estimate the last integral and we finally obtain
that there exists $0< \eta <1$ such that, for all $\gamma>0,$
\begin{eqnarray} \label{t3}
\mathcal{T}^3 &\leq& C_1 \mu \psi'(|a|)(\phi(x) + \phi(y))\int_{\mathcal{C}_{\eta, \gamma}(a)} \phi(z)|z|^2 \nu(dz) \\
&&  + C_1(\psi(|a|) + \delta)(\mathcal{I}[\mathcal{C}_{\eta, \gamma}(a)](x,\phi,D\phi) + \mathcal{I}[\mathcal{C}_{\eta, \gamma}(a)](y,\phi,D\phi))\nonumber\\
 && + \frac{1}{2}\Phi(x,y)  \int_{\mathcal{C}_{\eta, \gamma}(a)} \sup_{|s| \leq 1} l(a,s,z) |z|^2\nu(dz),\nonumber
\end{eqnarray}
where
\begin{eqnarray}\label{la}
l(a,s,z) = (1- \tilde{\eta}^2)\frac{\psi'(|a+sz|)}{|a+sz|} + \tilde{\eta}^2 \psi''(|a+sz|)
\end{eqnarray}
and $\tilde{\eta} = \frac{1 - \eta - \gamma_0}{1+\gamma_0},$ $\gamma = \gamma_0|a|$ with $\gamma_0 \in (0,1)$.
Notice that, if Assumption~\eqref{M2-assumpt} holds for $\eta,$ then it also holds for smaller $\eta$,
so we can choose $\eta$ as small as we want. Moreover, using $\eqref{M1}$  we get
\begin{eqnarray}\label{holder-bound}
&& C_1 \mu \psi'(|a|)(\phi(x) + \phi(y))\int_{\mathcal{C}_{\eta, \gamma}(a)} \phi(z)|z|^2 \nu(dz) 
\leq \mu C^1_\nu\psi'(|a|) \Phi(x,y).
\end{eqnarray}

Finally, from \eqref{t1t2}, \eqref{t3} and \eqref{holder-bound} we obtain
\begin{eqnarray*}\label{itotal}
\mathcal{T} &\leq& C_1(\psi(|a|) + \delta)\left(\mathcal{I}(x,\phi,D\phi) + \mathcal{I}(y,\phi,D\phi) \right)+ \mu C^1_\nu\psi'(|a|) \Phi(x,y)\\
&& + \frac{1}{2}\Phi(x,y)  \int_{\mathcal{C}_{\eta, \gamma}(a)} \sup_{|s| \leq 1} l(a,s,z)|z|^2\nu(dz).
\end{eqnarray*}
\end{proof}


\begin{proof}[\textbf{Proof of Lemma \ref{Liplem}}]
Let $a_0 > 0$, $|x-y|=|a| \leq a_0$. From Proposition \ref{nlocalgen} \eqref{ellnon}, we have
\begin{eqnarray*}
&&\mathcal{I}(x,u,D_x\varphi) - \mathcal{I}(y,u,-D_y\varphi)\\
&\leq&  C_1(\psi(|a|) + \delta)\left(\mathcal{I}(x,\phi,D\phi) + \mathcal{I}(y,\phi,D\phi) \right)+ \mu C^1_\nu\psi'(|a|) \Phi(x,y)\nonumber\\
&& + \frac{1}{2}\Phi(x,y)  \int_{\mathcal{C}_{\eta, \gamma}(a)} \sup_{|s| \leq 1} l(a,s,z)|z|^2\nu(dz),\nonumber
\end{eqnarray*}
where $l(a,s,z)$ is given by \eqref{la}. 

Let $\psi$ be defined in \eqref{psinon}, take $ a_0=  r_0$. It follows from \cite[Corollary 9]{bcci12} that there exists a constant $C = C(\nu)>0$ such that for $\Lambda(\nu) = C(\varrho\theta 2^{\theta-1}-1)> 0$ we have
\begin{eqnarray*}\label{lip-bound}
\frac{1}{2}\int_{\mathcal{C}_{\eta, \gamma}(a)} \sup_{|s| \leq 1} l(a,s,z)|z|^2\nu(dz)
\leq -\Lambda|a|^{-\tilde{\theta}},
\end{eqnarray*}
where $\tilde{\theta} = \beta-1 - \theta(N+2-\beta)>0$. 

Therefore, we obtain
\begin{eqnarray}
 \mathcal{I}(x,u,D_x\varphi) - \mathcal{I}(y,u,-D_y\varphi) 
&\leq& C_1(\psi(|a|) + \delta)\left(\mathcal{I}(x,\phi,D\phi) + \mathcal{I}(y,\phi,D\phi) \right)
\nonumber\\
&&
 - \left(\Lambda |a|^{-\tilde{\theta}} - \mu C_\nu^1 \psi'(|a|) \right)\Phi(x,y).\nonumber
\end{eqnarray}
\end{proof}

\begin{proof}[\textbf{Proof of Lemma \ref{holder-lem}.}]
Let $a_0 > 0$, $|x-y|=|a| \leq a_0$. From Proposition \ref{nlocalgen} \eqref{ellnon} we have
\begin{eqnarray}\label{holder-estimates}
 &&\mathcal{I}(x,u,D_x\varphi) - \mathcal{I}(y,u,-D_y\varphi)\\
&\leq& C_1(\psi(|a|) + \delta)\left(\mathcal{I}(x,\phi,D\phi) + \mathcal{I}(y,\phi,D\phi) \right)
\nonumber\\
&& + \left(\mu C^1_\nu\psi'(|a|) +  \frac{1}{2} \int_{\mathcal{C}_{\eta, \gamma}(a)} \sup_{|s| \leq 1} l(a,s,z)|z|^2\nu(dz)\right)\Phi(x,y),\nonumber
\end{eqnarray}
where $l(a,s,z)$ is given by \eqref{la}. We only need to estimate this terms and then integrate over the cone.

Let $r=|x-y|,$ $\psi(r) = 1 - e^{-C_2r^\tau}$ for $r \leq r_0$, $\tau \in (0,1)$, we have the derivatives
\begin{eqnarray*}
&& \psi'(r) = C_2\tau r^{\tau -1}e^{-C_2r^\tau},~~
\psi''(r)= C_2\tau (\tau -1)r^{\tau -2}e^{-C_2r^\tau} - (C_2\tau r^{\tau -1})^2e^{-C_2r^\tau}.
\end{eqnarray*}
Hence, we have
\begin{eqnarray*}
l(a,s,z)& =& (1- \tilde{\eta}^2) C_2 \tau e^{-C_2|a+sz|^\tau}|a+sz|^{\tau-2} + \tilde{\eta}^2 C_2 \tau (\tau -1) \tau e^{-C_2|a+sz|^\tau}|a+sz|^{\tau-2}\\
&& \qquad  - \tilde{\eta}^2 (C_2 \tau)^2 e^{-C_2|a+sz|^\tau} |a+sz|^{2(\tau-1)}\\
&=&  C_2 \tau e^{-C_2|a+sz|^\tau}|a+sz|^{\tau-2} \left( 1 -  \tilde{\eta}^2 (2 - \tau)\right)  - \tilde{\eta}^2 (C_2 \tau)^2 e^{-C_2|a+sz|^\tau} |a+sz|^{2(\tau-1)}.
\end{eqnarray*}
Note that, on the set $\mathcal{C}_{\eta,\gamma}(|a|),$ we have the following upper bound
$$|a + sz| \leq |a| + |s||z| \leq |a| + \gamma = |a|(1+\gamma_0).$$
Taking $\tau \in (0,1)$ (possibly arbitrary close to 1) then $ 2 - \tau > 1$. So we can choose $\eta$ and $\gamma_0$ sufficiently small enough such that 
\begin{eqnarray*}
&&(2-\tau)\tilde{\eta}^2 = (2 - \tau)\left(\frac{1-\eta - \gamma_0}{1+\gamma_0}\right)^2 > \theta >1 ~~ \text{for some } \theta \in (1,2-\tau).\end{eqnarray*}
Therefore we obtain
\begin{eqnarray*}
l(a,s,z)&\leq& -  C_2\tau (\theta - 1)  e^{-C_2|a+sz|^\tau}|a+sz|^{\tau-2} -  (C_2 \tau)^2 \frac{\theta}{2 - \tau}e^{-C_2|a+sz|^\tau} |a+sz|^{2(\tau-1)}\\
&\leq& -  C_2 \tau (\theta - 1) e^{-C_2|a|^\tau}e^{-C_2|z|^\tau} (1+\gamma_0)^{\tau-2}|a|^{\tau -2} \\
 && \qquad -  (C_2 \tau)^2 \frac{\theta}{2 - \tau}e^{-C_2|a|^\tau}e^{-C_2|z|^\theta} (1+\gamma_0)^{2(\tau-1)}|a|^{2(\tau-1)}\\
&=& - \psi'(|a|) \left( C^1(\nu,\tau)|a|^{-1} + C_2 C^2(\nu,\tau)|a|^{\tau -1} \right) e^{-C_2|z|^\tau}.
\end{eqnarray*}
Remark that $\eta, \gamma_0$ do not depend on $|a|$. Taking the integral for $l(a,s,z)$ over the cone $\mathcal{C}_{\eta,\gamma}(a)$ for $\gamma = \gamma_0 |a|$ and using $\eqref{M2-assumpt}$ we obtain
\begin{eqnarray}\label{holder-ell}
&&\frac{1}{2} \int_{\mathcal{C}_{\eta, \gamma}(a)} \sup_{|s| \leq 1} l(a,s,z)|z|^2 \nu(dz)\\
&\leq& - \frac{1}{2}\psi'(|a|) \left( C^1(\nu,\tau)|a|^{-1} + C_2 C^2(\nu,\tau)|a|^{\tau -1} \right)\int_{\mathcal{C}_{\eta, \gamma}(a)}e^{-C_2|z|^\tau}|z|^2 \nu(dz) \nonumber\\
&\leq& - C(\nu,\tau) \psi'(|a|) (|a|^{-1} +C_2 |a|^{\tau -1})|a|^{2-\beta}\nonumber\\
&=& - C(\nu,\tau) \psi'(|a|)(1+ C_2 |a|^\tau)|a|^{1-\beta},\nonumber
\end{eqnarray}
where $C(\nu,\tau) = \min\{C^1(\nu, \tau), C^2(\nu,\tau)\}C^2_\nu$, $C^2_\nu$ is in $\eqref{M2-assumpt}$.

Therefore, from \eqref{holder-estimates} and \eqref{holder-ell} we obtain
\begin{eqnarray*}
 \mathcal{I}(x,u,D_x\varphi) - \mathcal{I}(y,u,-D_y\varphi)
&\leq& C_1(\psi(r) + \delta)\left(\mathcal{I}(x,\phi,D\phi) + \mathcal{I}(y,\phi,D\phi) \right)
\nonumber\\
&& + \Phi(x,y)\left( \mu C^1_\nu - C(\nu,\tau)(1 + C_2 r^\tau)r^{1-\beta} \right) \psi'(r).
\end{eqnarray*}
\end{proof}


\begin{thebibliography}{10}

\bibitem{at96}
Olivier Alvarez and Agn{\`e}s Tourin.
\newblock Viscosity solutions of nonlinear integro-differential equations.
\newblock {\em Ann. Inst. H. Poincar\'e Anal. Non Lin\'eaire}, 13(3):293--317,
  1996.

\bibitem{bcg15}
Martino Bardi, Annalisa Cesaroni, and Daria Ghilli.
\newblock Large deviations for some fast stochastic volatility models by
  viscosity methods.
\newblock {\em Discrete Contin. Dyn. Syst.}, 35(9):3965--3988, 2015.

\bibitem{barles91b}
G.~Barles.
\newblock Interior gradient bounds for the mean curvature equation by viscosity
  solutions methods.
\newblock {\em Differential Integral Equations}, 4(2):263--275, 1991.

\bibitem{barles91a}
G.~Barles.
\newblock A weak {B}ernstein method for fully nonlinear elliptic equations.
\newblock {\em Differential Integral Equations}, 4(2):241--262, 1991.

\bibitem{bbl02}
G.~Barles, S.~Biton, and O.~Ley.
\newblock A geometrical approach to the study of unbounded solutions of
  quasilinear parabolic equations.
\newblock {\em Arch. Ration. Mech. Anal.}, 162(4):287--325, 2002.

\bibitem{bs01}
G.~Barles and P.~E. Souganidis.
\newblock Space-time periodic solutions and long-time behavior of solutions to
  quasi-linear parabolic equations.
\newblock {\em SIAM J. Math. Anal.}, 32(6):1311--1323 (electronic), 2001.

\bibitem{bcci12}
Guy Barles, Emmanuel Chasseigne, Adina Ciomaga, and Cyril Imbert.
\newblock Lipschitz regularity of solutions for mixed integro-differential
  equations.
\newblock {\em J. Differential Equations}, 252(11):6012--6060, 2012.

\bibitem{bcci14}
Guy Barles, Emmanuel Chasseigne, Adina Ciomaga, and Cyril Imbert.
\newblock Large time behavior of periodic viscosity solutions for uniformly
  parabolic integro-differential equations.
\newblock {\em Calc. Var. Partial Differential Equations}, 50(1-2):283--304,
  2014.

\bibitem{bi08}
Guy Barles and Cyril Imbert.
\newblock Second-order elliptic integro-differential equations: viscosity
  solutions' theory revisited.
\newblock {\em Ann. Inst. H. Poincar\'e Anal. Non Lin\'eaire}, 25(3):567--585,
  2008.

\bibitem{bklt15}
Guy Barles, Shigeaki Koike, Olivier Ley, and Erwin Topp.
\newblock Regularity results and large time behavior for integro-differential
  equations with coercive {H}amiltonians.
\newblock {\em Calc. Var. Partial Differential Equations}, 54(1):539--572,
  2015.

\bibitem{blt17}
Guy Barles, Olivier Ley, and Erwin Topp.
\newblock Lipschitz regularity for integro-differential equations with coercive
  {H}amiltonians and applications to large time behavior.
\newblock {\em Nonlinearity}, 30:703--734, 2017.

\bibitem{bt16b}
Guy Barles and Erwin Topp.
\newblock Lipschitz regularity for censored subdiffusive integro-differential
  equations with superfractional gradient terms.
\newblock {\em Nonlinear Anal.}, 131:3--31, 2016.

\bibitem{clp10}
I.~Capuzzo~Dolcetta, F.~Leoni, and A.~Porretta.
\newblock H\"older estimates for degenerate elliptic equations with coercive
  {H}amiltonians.
\newblock {\em Trans. Amer. Math. Soc.}, 362(9):4511--4536, 2010.

\bibitem{cil92}
M.~G. Crandall, H.~Ishii, and P.-L. Lions.
\newblock User's guide to viscosity solutions of second order partial
  differential equations.
\newblock {\em Bull. Amer. Math. Soc. (N.S.)}, 27(1):1--67, 1992.

\bibitem{dpv12}
Eleonora Di~Nezza, Giampiero Palatucci, and Enrico Valdinoci.
\newblock Hitchhiker's guide to the fractional {S}obolev spaces.
\newblock {\em Bull. Sci. Math.}, 136(5):521--573, 2012.

\bibitem{fs93}
W.~H. Fleming and H.~M. Soner.
\newblock {\em Controlled {M}arkov processes and viscosity solutions}.
\newblock Springer-Verlag, New York, 1993.

\bibitem{fil06b}
Y.~Fujita, H.~Ishii, and P.~Loreti.
\newblock Asymptotic solutions of {H}amilton-{J}acobi equations in {E}uclidean
  {$n$} space.
\newblock {\em Indiana Univ. Math. J.}, 55(5):1671--1700, 2006.

\bibitem{fil06}
Y.~Fujita, H.~Ishii, and P.~Loreti.
\newblock Asymptotic solutions of viscous {H}amilton-{J}acobi equations with
  {O}rnstein-{U}hlenbeck operator.
\newblock {\em Comm. Partial Differential Equations}, 31(4-6):827--848, 2006.

\bibitem{fl09}
Y.~Fujita and P.~Loreti.
\newblock Long-time behavior of solutions to {H}amilton-{J}acobi equations with
  quadratic gradient term.
\newblock {\em NoDEA Nonlinear Differential Equations Appl.}, 16(6):771--791,
  2009.

\bibitem{ghilli16}
Daria Ghilli.
\newblock Viscosity methods for large deviations estimates of multiscale
  stochastic processes.
\newblock {\em Submitted}, 2016.

\bibitem{gt83}
D.~Gilbarg and N.~S. Trudinger.
\newblock {\em Elliptic partial differential equations of second order}.
\newblock Springer-Verlag, Berlin, second edition, 1983.

\bibitem{il90}
H.~Ishii and P.-L. Lions.
\newblock Viscosity solutions of fully nonlinear second-order elliptic partial
  differential equations.
\newblock {\em J. Differential Equations}, 83(1):26--78, 1990.

\bibitem{koike04}
S.~Koike.
\newblock {\em A beginner's guide to the theory of viscosity solutions},
  volume~13 of {\em MSJ Memoirs}.
\newblock Mathematical Society of Japan, Tokyo, 2004.

\bibitem{ln16}
O.~Ley and V.~D. Nguyen.
\newblock Gradient bounds for nonlinear degenerate parabolic equations and
  application to large time behavior of systems.
\newblock {\em Nonlinear Anal.}, 130:76--101, 2016.

\bibitem{ln17}
O.~Ley and V.~D. Nguyen.
\newblock Lipschitz regularity results for nonlinear strictly elliptic
  equations and applications.
\newblock {\em Submitted}, 2016.

\bibitem{lions82}
P.-L. Lions.
\newblock {\em Generalized solutions of {H}amilton-{J}acobi equations}.
\newblock Pitman (Advanced Publishing Program), Boston, Mass., 1982.

\bibitem{ls05}
P.-L. Lions and P.~E. Souganidis.
\newblock Homogenization of degenerate second-order {PDE} in periodic and
  almost periodic environments and applications.
\newblock {\em Ann. Inst. H. Poincar\'e Anal. Non Lin\'eaire}, 22(5):667--677,
  2005.

\bibitem{nguyen17}
T.~T. Nguyen.
\newblock Large time behavior of solutions of local and nonlocal nondegenerate
  {H}amilton-{J}acobi equations with {O}rnstein-{U}hlenbeck operator.
\newblock {\em In preparation}, 2017.

\end{thebibliography}

\end{document}